\documentclass[11pt]{article}
\usepackage[utf8]{inputenc}

\usepackage{amstext,amssymb,amsmath,amsbsy}

\usepackage{hyperref}
\usepackage{amscd}
\usepackage{amsfonts}
\usepackage{indentfirst}
\usepackage{verbatim}
\usepackage{amsmath}
\usepackage{amsthm}
\usepackage{enumerate}
\usepackage{graphicx}
\usepackage{color, soul}
\usepackage[OT1]{fontenc}
\usepackage[english]{babel}
\usepackage{amssymb}
\usepackage{subfig}
\usepackage{algorithm}
\usepackage{algorithmic}

\usepackage{mathtools}
\DeclarePairedDelimiter\ceil{\lceil}{\rceil}

\newcommand{\R}{\mathbb{R}}

\newcommand{\x}{{\bf x}}

\newcommand{\p}{{\bf p}}

\newtheorem{Theorem}{Theorem}[section]
\newtheorem{Lemma}{Lemma}[section]

\newtheorem{Corollary}{Corollary}[section]
\newtheorem{Remark}{Remark}[section]

\newtheorem{Assumption}{Assumption}[section]
\newtheorem*{Assumption*}{Assumption}
\newtheorem{Definition}{Definition}[section]
\newtheorem{Problem}{Problem}[section]
\newtheorem*{problem*}{Problem}
\numberwithin{equation}{section}

\title{The Carleman-Newton  method to globally reconstruct a source term for nonlinear parabolic equation}

\author{Anuj Abhishek\thanks{Department of Mathematics and Statistics, University of North Carolina at
  Charlotte, Charlotte, NC, 28223, USA.} \thanks{Corresponding author, email: \texttt{anuj.abhishek@uncc.edu}.}  \and Thuy T. Le\footnotemark[1] \and Loc H. Nguyen\footnotemark[1]
  \and Taufiquar Khan\footnotemark[1]}

\date{}

\begin{document}

\maketitle

\begin{abstract}

We propose to combine the Carleman estimate and the Newton method to solve an inverse source problem for nonlinear parabolic equations from lateral boundary data.
The stability of this inverse source problem is  conditionally logarithmic.
Hence, numerical results due to the conventional least squares optimization might not be reliable.
In order to enhance the stability, we approximate this problem by truncating the high frequency terms of the Fourier series that represents the solution to the governing equation. 
By this, we derive a system of nonlinear elliptic PDEs whose solution consists of Fourier coefficients of the solution to the parabolic governing equation.
We solve this system by the  Carleman-Newton method. 
The Carleman-Newton method is a newly developed algorithm to solve nonlinear PDEs.
The strength of the Carleman-Newton method includes (1) no good initial guess is required and (2) the computational cost is not expensive.
These features are rigorously proved.
Having the solutions to this system in hand, we can directly compute the solution to the proposed inverse problem.
Some numerical examples are displayed.
\end{abstract}



\section{Introduction}

\label{sec Intro}

This paper belongs to a series of works  to solve inverse problems for nonlinear partial differential equations \cite{LeNguyen:jiip2022, NguyenNguyenTruong:arxiv2022, NguyenKlibanov:ip2022}. 
In particular, we aim to globally solve an inverse source problem for nonlinear parabolic equations.
By ``globally", we mean that our algorithm does not require  {\it a priori} knowledge of the true solution. 
This feature is a significant strength of our method in comparison to the widely-used methods based on least square optimizations, which are locally convergent.
Let $d \geq 2$ be the spatial dimension.
Let $F: \R^d \times \R  \times \R^d \to \R$ be a function in the class $C^1(\overline \Omega)$ and $T$ be a positive number.
Consider the following initial value problem of nonlinear parabolic equation
\begin{equation}
    \left\{
        \begin{array}{ll}
            u_t = \Delta u + F(\x, u, \nabla u)  &(\x, t) \in \R^d \times (0, T),\\
            u(\x, 0) = p(\x) &\x \in \R^d.
        \end{array}
    \right.
    \label{main_eqn}
\end{equation}
Here, $p$ is a source term.
Let $\Omega$ be an open and bounded domain of $\R^d$. Assume that $\Omega$ has a smooth boundary $\partial \Omega.$ 
Denote by $\Omega_T$ and $\partial \Omega_T$ the set $\Omega \times (0, T)$ and $\partial \Omega \times (0, T)$ respectively.
We are interested in the problem of computing the source term $p$ from the measurement of lateral information of the function $u$. 
More precisely, we solve the following problem.
\begin{Problem}[Inverse source problem]
    Given the lateral data 
    \begin{equation}
        g_0(\x, t) = u(\x, t)
        \quad\mbox{and}
        \quad
        g_1(\x, t) = \partial_{\nu }u(\x, t)
        \label{data}
    \end{equation} 
    for all $(\x, t) \in \partial \Omega_T,$ 
    determine the source term $p(\x)$ for $\x \in \Omega.$
    \label{isp}
\end{Problem}

We only solve Problem \ref{isp} under the condition that \eqref{main_eqn} is uniquely solvable and its solution is bounded in $C^1$. 

More precisely, we impose the following condition on solution to \eqref{main_eqn}.
\begin{Assumption}
    Assume the source function $p$ and the nonlinearity $F$ are such that  
    \begin{equation}
        |u(\x, t)| + |\nabla u(\x, t)| \leq M
        \label{1.2}
    \end{equation}
    for a.e. $(\x, t) \in \Omega_T$ and for some number $M$ depending only on $p$ and $F.$
    \label{Assump}
\end{Assumption}

In general, due to the presence of the nonlinearity $F$, the well-posedness and regularity results for \eqref{main_eqn} are not guaranteed. In other words, Assumption \ref{Assump} might not hold true.
Some special conditions on $F$ and $p$ should be imposed.
For completeness, we provide here a set of conditions that guarantees that \eqref{main_eqn} has a unique and bounded solution.
Assume that  $p(\x)$ is in $H^{2+\alpha}(\R^d)$ for some $\alpha \in [0, 1 + 4/d]$. Assume further  for all 
$\x \in \R^d$, $t \in [0, T]$, $s \in \R$, and $\xi \in \R^d$,
\begin{equation}
	|F(\x,  s, \xi)| \leq C\max\big\{
		(1 + |\xi|)^2, 1 + |s|
	\big\}
	\label{Lady}
\end{equation}
for some positive constant $C$.
Then, due to Theorem 6.1 in \cite[Chapter 5, \S 6]{LadyZhenskaya:ams1968} and Theorem 2.1 in \cite[Chapter 5, \S 2]{LadyZhenskaya:ams1968}, 
problem \eqref{main_eqn} has a unique solution with a bounded $C^1$ norm. 

\begin{Definition}
    Fix a nonlinearity $F$. Denote by  $\mathcal{P}$ the set of all source functions $p \in H^1(\Omega)$ such that Assumption \ref{Assump} holds true.
    \label{defP}
\end{Definition}

In practice, the solution $p(\x),$ $\x \in \Omega,$ to
Problem \ref{isp} might represent the initial distribution of the temperature or the initial stage of the pollutant. 
Therefore, computing $p$ is important in 
 many real-world applications, e.g.,	 
determination of the spatially distributed temperature inside a solid from the boundary measurement of the heat and heat flux in the time domain \cite{Klibanov:ip2006}; 
determining the level of pollutant on the surface of the rivers or lakes \cite{BadiaDuong:jiip2002};
 effective monitoring the heat conduction processes in steel industries, glass and polymer-forming and nuclear power station \cite{LiYamamotoZou:cpaa2009}.
In the special case when the nonlinear term $F$ takes the form $u(1 - u)$ (or $q(u) = u(1 - |u|^{\alpha})$) for some $\alpha > 0$, the parabolic equation in \eqref{main_eqn} is called the high dimensional version of the well-known Fisher (or Fisher-Kolmogorov) equation \cite{Fisher:ae1937}.
Although the nonlinearity $q$ does not satisfy condition \eqref{Lady}, we do not experience any difficulty in numerical computations of the forward problem.
It is worth mentioning that the Fisher equation occurs in ecology, physiology, combustion, crystallization, plasma physics, and in general phase transition problems, see \cite{Fisher:ae1937}.
Due to its realistic applications, the problem of determining the initial conditions of parabolic equations has been studied  intensively. 
The uniqueness of Problem \ref{isp} for linear model was proved in  \cite{Lavrentiev:AMS1986}.
The the logarithmic stability results were rigorously proved in \cite{Klibanov:ip2006, LiYamamotoZou:cpaa2009}.
Due to the presence of the nonlinear term $F(\x, u, \nabla u)$ in \eqref{main_eqn}, the uniqueness and stability of Problem \ref{isp} might need to be proved, 
especially when the nonlinearity $F$ might growth faster than the nonlinearity in \cite{Klibanov:ip2006, LiYamamotoZou:cpaa2009}.
This task can be done by combining Theorem 1 in \cite{Klibanov:ip2006} and a truncation technique. 
Since the proof is not complicated, we present this result here in this paper for the convenience of the reader.

Regarding constructive method, the widely used methods to solve inverse source problem like Problem \ref{isp} are based on optimization.
In such methods, one typically optimizes a cost functional based on the problem.
A prototypical example of a cost functional is
\[
    J(p) = \int_0^T\int_{\partial \Omega}\big[|L(p)(\x, t) - g_0(\x, t)|^2 + |\partial_{\nu}L(p)(\x, t) - g_1(\x, t)|^2
    \big] d\sigma(\x) dt
    + \mbox{a regularization term}
\]
where $L(p) = u$ is the solution to \eqref{main_eqn}.
The minimizer of the cost functional serves as the computed solution to the inverse problem.
Using this technique is challenging because inverse source problem for parabolic equation is severely ill-posed. 
The known stability is conditionally logarithmic (see  \cite[Theorem 1]{Klibanov:ip2006} and \cite{LiYamamotoZou:cpaa2009}), which is weaker than logarithmic. 
As a result, the reliability of numerical solutions due to these approaches might not be guaranteed, especially,  in the case when significant noise is involved in the measured data.
Another challenge in solving Problem \ref{isp} is the presence of the nonlinear term $F(\x, u, \nabla u),$ making the cost functional nonconvex.
The cost functional might have multiple local minima and ravines. 
Hence, in general, to solve Problem \ref{isp}, a good initial guess of the true solution is necessary.
On the other hand, the expensive computational cost is another drawback of the optimization-based methods.
Our new numerical approach proposed in this paper relaxes all drawbacks above. 
It is applicable to the data containing up to $20\%$ of (multiplicative) noise.
 We name it Carleman-Newton method because it is suggested by  Carleman estimates and the well-know Newton method in solving nonlinear equations.

We draw the reader's attention to the Carleman contraction principle, see \cite{BAUDOUIN:SIAMNumAna:2017, Boulakia:esaim2021, LeNguyen:jiip2022, Le:preprintCONN2022} to solve related problems to Problem \ref{isp}.
Like the proposed Carleman-Newton method, the Carleman contraction principle is very powerful since it can quickly solve nonlinear inverse problems without requesting a good initial guess. 
The new point of our paper in comparison to the cited publications above is that we allow $F$ to depend on $\nabla u$ and we study the case when data contains noise. 
The case when $F$ depends on $\nabla u$ arises from the field of Hamilton-Jacobi equations of the form $u_t = \epsilon \Delta u + F(\x, u, \nabla u)$, $0 < \epsilon \ll 1$ where the term $\epsilon \Delta u$ serves as the viscosity term. Hamilton-Jacobi equations are important in many fields; e.g.  game theory and light propagation. 
Therefore, this gradient dependent case is worth studying.
Rather than using the Carleman contraction principle as in \cite{BAUDOUIN:SIAMNumAna:2017, Boulakia:esaim2021, LeNguyen:jiip2022, Le:preprintCONN2022}, we develop the Carleman-Newton method to solve the inverse problem under consideration. 
The Carleman-Newton method in this paper is stronger than the one introduced in \cite{LeNguyenTran:preprint2021}.
In fact, we can prove a Lipschitz stability result of our proposed method with respect to the $H^2$ norm while the convergence in the cited papers above, as well as \cite{LeNguyenTran:preprint2021}, is with respect to the $H^1$ norm.
Another new point of this paper is that in the current paper, we study the noise analysis, which was missing in \cite{LeNguyenTran:preprint2021}.

As mentioned, the best stability result for Problem \ref{isp} is conditionally logarithmic.
Therefore, solving it is extremely challenging. 
To overcome this difficulty, we propose to solve Problem \ref{isp} in the Fourier domain truncating all high frequency components. 
More precisely, we derive a system of elliptic PDEs whose solution consists of a finite number of the Fourier coefficients of the solution to the parabolic equation (\ref{main_eqn}). 
The solution of this system directly yields the knowledge of the function $u(\x, t)$, from which the solution to our inverse problem follows. 
We numerically solve this nonlinear system by the Carleman-Newton method suggested in \cite{LeNguyenTran:preprint2021}. 
The initial solution can be computed by solving the system obtained by removing the nonlinear term.
Then, we approximate the nonlinear system by its linearization. 
Solving this approximation system, we find an updated solution.
Continuing this process, we get a fast convergent sequence reaching to the desired function.
The convergence of this iterative procedure is rigorously proved by using a new Carleman estimate.
The fast convergence will be shown in both analytic and numerical senses.

Some papers closely related to the current one are \cite{Boulakia:esaim2021, LeNguyen:jiip2022, Le:preprintCONN2022, LiNguyen:IPSE2020}.
On the other hand, the coefficient inverse problem for parabolic equations is also very interesting and studied intensively. 
We draw the reader's attention to \cite{Borceaetal:ip2014, CaoLesnic:nmpde2018, CaoLesnic:amm2019, KeungZou:ip1998, Nguyen:CAMWA2020, Nguyens:jiip2020, YangYuDeng:amm2008} for important numerical methods and good numerical results.
Besides, the  problem of recovering the initial conditions for the hyperbolic equation is very interesting since it arises in many real-world applications. 
For instance, the problems thermo- and photo-acoustic tomography play the key roles in biomedical imaging. 
We refer the reader to some important works in this field \cite{HaltmeierNguyen:SIAMJIS2017,
KatsnelsonNguyen:aml2018,
LiuUhlmann:ip2015}. 
Applying the Fourier transform, one can reduce the problem of reconstructing the initial conditions for hyperbolic equations to some inverse source problems for the Helmholtz equation, see \cite{
LiLiuSun:IPI2018,
NguyenLiKlibanov:2019, 
WangGuoLiLiu:ip2017,
WangGuoZhangLiu:ip2017, ZhangGuoLiu:ip2018} for some recent results.

The paper is organized as follows. 
In Section \ref{sec2}, we prove the uniqueness and the conditionally logarithmic stability results for Problem \ref{isp}.
In Section \ref{sec_method}, we introduce an approximation context of Problem \ref{isp}.
In Section \ref{sec_CarNew}, we recall the Carleman-Newton method in \cite{LeNguyenTran:preprint2021}.
In Section \ref{sec_convergence}, we prove our main theorem about the convergence of the Carleman-Newton method.
In Section \ref{NS}, we present some numerical examples.
Section \ref{sec7} is for concluding remarks.

\section{The uniqueness of the inverse source problem}\label{sec2}

In this section, we study the unique determination of the source function $p$ from the given data. 
We have the following theorem.

\begin{Theorem}
    Let $p_1$ and $p_2$ be in the set $\mathcal{P}$.
    Let $u_1$ and $u_2$ be solutions to \eqref{main_eqn} with $p$ being replaced by $p_1$ and $p_2$ respectively.
    Let $M_1$ and $M_2$ be the numbers in the right hand side of \eqref{1.2} that correspond to $p_1$ and $p_2$ respectively.  
    Let $B = \max\{M_1, M_2\}.$
    Then, there exists a constant $C$ such that for all $\beta \in (0, 2)$, we can find a number $\epsilon_0 > 0$ such that 
    \begin{equation}
        \|p_1 - p_2\|_{L^2(\Omega)} \leq \frac{C}{\beta \ln[\frac{B}{\epsilon_0 {\bf e}}]} \|\nabla (p_1 - p_2)\|_{L^2(\Omega)} + C \Big(\frac{B}{\epsilon_0}\Big)^\beta {\bf e}^{2 - \beta}
        \label{log}
    \end{equation}
    where
    \begin{equation}
        {\bf e} = \|u_1 - u_2\|_{H^1(\partial \Omega_T)} + \|\partial_{\nu}(u_1 - u_2)\|_{L^2(\partial \Omega_T)}.
        \label{bfe}
    \end{equation}
    represents the difference of two data corresponding to two source functions $p_1$ and $p_2$.
    \label{thm_uni}
\end{Theorem}
\begin{Corollary}[The uniqueness of Problem \ref{isp}]
    It follows from \eqref{log} that by letting ${\bf e}$ tend to $0$, we obtain $p_1 = p_2,$ which is basically the uniqueness of Problem \ref{isp}. 
    \label{col_uni}
\end{Corollary}

\begin{proof}[Proof of Theorem \ref{thm_uni}]
Let $\chi_B: \overline \Omega \times \R \times \R^d \to [0, 1]$ be a  cut off function in the class $C^{\infty}$ that satisfies
\begin{equation}
    \chi_B(\x, s, \p) =
    \left\{
        \begin{array}{ll}
            1 & 0 \leq |s| + |\p| \leq B\\
            0 & |s| + |\p| > 2B.
        \end{array}
    \right.
    \label{chi_B}
\end{equation}
Define 
\begin{equation}
    F_B(\x, s, \p) = \chi_B(\x, s, \p) F(\x, s, \p),
    \quad
    \mbox{for all }
    (\x, s, \p) \in \overline \Omega \times \R \times \R^d.
    \label{2,3}
\end{equation} 
Since Assumption \ref{Assump} holds true for $p_1$ and $p_2$,
we have
\[
    |u_i(\x, t)| + |\nabla u_i(\x, t)| \leq 
    M_i \leq B
\]
for $(\x, t) \in \overline{\Omega_T}$, $i \in \{1, 2\}.$
Hence, both $u_1$ and $u_2$ satisfy the ``cut off" parabolic equation
\begin{equation}
    u_t = \Delta u + F_B(\x, u, \nabla u)
    \quad
    \mbox{for all } (\x, t) \in \overline{\Omega_T}.
    \label{2.1}
\end{equation}
    Since $F_B$ is smooth and has compact support, it is Lipschitz.
There exists a constant $C_{F, B}$ depending only on $F$ and $B$ such that
\begin{equation}
    |F_B(\x, s_1, \p_1) - F_B(\x, s_2, \p_2)| \leq C_{F, B} (|s_1 - s_2| + |\p_1 - \p_2|)
    \label{Lip}
\end{equation}
for all $\x \in \overline \Omega$, $s_1, s_2 \in \R$ and $\p_1, \p_2 \in \R^d.$
Define $h = u_1 - u_2.$
Due to \eqref{2.1} and \eqref{Lip},
\begin{align}
    |h_t(\x, t) - \Delta h(\x, t)| 
    &= |F_B(\x, u_1(\x, t), \nabla u_1(\x, t)) - |F_B(\x, u_2(\x, t), \nabla u_2(\x, t))|
    \nonumber
    \\
    &\leq C_{F, B} (|h(\x, t)| + |\nabla h(\x, t)| \label{2.3}
\end{align}
for all $(\x, t) \in \overline{\Omega_T}.$
It is obvious that
\[
    {\bf e} = \|h\|_{H^1(\partial \Omega_T)} + \|\partial_{\nu} h\|_{L^2(\partial \Omega_T)}.
\]
Fix $\beta \in (0, 2)$.
Using \eqref{2.3} and applying Theorem 1 in \cite{Klibanov:ip2006} for the function $h$, we can find a constant $C > 0$ and  a number $\epsilon_0 \in (0, 1)$ such that
\begin{equation}
    \|h(\x, 0)\|_{L^2(\Omega)} \leq
    \frac{C}{\beta \ln[\frac{B}{\epsilon_0 {\bf e}}]} \|\nabla h(\x, 0)\|_{L^2(\Omega)} + C \Big(\frac{B}{\epsilon_0}\Big)^\beta {\bf e}^{2 - \beta}.
    \label{2.6}
\end{equation}
Estimate \eqref{log} is a direct consequence of \eqref{2.6}.
\end{proof}

\begin{Remark}
    Although estimate \eqref{log} guarantees the uniqueness of Problem \ref{isp} (see Corollary \ref{col_uni}), it does not lead to a reliable numerical approach to solve Problem \ref{isp}.
    In fact, due to the presence of the term $\|\nabla (p_1 - p_2)\|_{L^2(\Omega)}$ in the right hand side of \eqref{log}, we cannot obtain a  stability result for Problem \ref{isp}.
    Therefore, methods to solve Problem \ref{isp} based on optimization might not provide reliable solutions. 
     \label{rem1} 
\end{Remark}

By Remark \ref{rem1}, rather than employing the optimization approach, to numerically solve Problem \ref{isp}, we propose to regularize and approximate the inverse problem by truncating the high frequency components of the solution to \eqref{main_eqn}. 
This idea was introduced in \cite{Klibanov:jiip2017}. Then, it was successfully used in many projects of our research group; see e.g. \cite{KhoaKlibanovLoc:SIAMImaging2020, LeNguyen:jiip2022,  LeNguyen:JSC2022}. Details will be given in the next section.


\section{A numerical method to solve Problem \ref{isp}} \label{sec_method}

Our method to stably solve Problem \ref{isp} consists of two stages. 
In stage 1, we derive a ``Garlekin-Fourier" approximation model of \eqref{main_eqn}. 
In stage 2, we solve that approximation model by the Carleman-Newton method, first introduced in \cite{LeNguyenTran:preprint2021}.
Stage 1 is presented in this section while we will develop a numerical method for Stage 2 in Sections \ref{sec_CarNew} and \ref{sec_convergence}.


Motivated by \cite{LeNguyen:jiip2022} and  Remark \ref{rem1}, we solve Problem \ref{isp} by establishing a system of quasi-linear partial differential equations with Cauchy boundary data. 
This system will be solved later by the method proposed in  \cite{LeNguyenTran:preprint2021}.
Let $\{\Psi_n\}_{n \geq 1}$ be an orthonormal basis of $L^2(0, T).$
We approximate the solution $u$ to \eqref{main_eqn} by truncating its Fourier expansion with respect to the basis $\{\Psi_n\}_{n \geq 1}$ as follows
\begin{equation}
    u(\x, t) = \sum_{n = 1}^\infty u_n(\x) \Psi_n(t) 
    \simeq \sum_{n = 1}^N u_n(\x) \Psi_n(t)
    \label{3.1}
\end{equation}
for $(\x, t) \in \overline{\Omega_T}$ where
\begin{equation}
    u_n(\x) = \int_0^T u(\x, t) \Psi_n(t)dt
    \label{3.2}
\end{equation}
for some ``cut off" number $N$. 
The cut off number $N$ will be numerically chosen later. 
We also approximate $u_t(\x, t)$, $(\x, t) \in \overline{\Omega_T}$, by
\begin{equation}
    u_t(\x, t) \simeq \sum_{n = 1}^N u_n(\x) \Psi_n'(t).
    \label{3.3}
\end{equation}
From now on, we assume that the approximations in \eqref{3.1} and \eqref{3.3} are valid in the sense that the resulting errors are sufficiently small. 
This assumption is acceptable in computation. 
It somewhat similar to the main principle in Galerkin approximation, in which we approximate the function $u(\x, \cdot)$ by using a finite number of elements in the orthonormal basis $\{\Psi_n\}_{n \geq 1}$ of $L^2(0, T).$ 
Plugging \eqref{3.1} and \eqref{3.3} into \eqref{main_eqn}, we obtain
\begin{equation}
    \sum_{n = 1}^N u_n(\x)\Psi_n'(t) = \sum_{n = 1}^N \Delta  u_n(\x) \Psi_n(t) + F\Big(\x, \sum_{n = 1}^N u_n(\x) \Psi_n(t), \sum_{n = 1}^N \nabla u_n(\x) \Psi_n(t)\Big)
    \label{3.4}
\end{equation}
for all $(\x, t) \in \overline{\Omega_T}.$
For each $m \in \{1, \dots, N\}$, we multiply the function $\Psi_m$ to both sides of \eqref{3.4} and the integrate the resulting equation with respect to $t$. 
We obtain
\begin{multline}
    \sum_{n = 1}^N u_n(\x) \int_0^T \Psi_n'(t)\Psi_m(t) dt 
    = \sum_{n = 1}^N \Delta  u_n(\x) \int_0^T \Psi_n(t)\Psi_m(t)dt
    \\
    + \int_0^T F\Big(\x, \sum_{n = 1}^N u_n(\x) \Psi_n(t), \sum_{n = 1}^N \nabla u_n(\x) \Psi_n(t)\Big)\Psi_m(t) dt
    \label{3.5}
\end{multline}
for all $\x \in \overline \Omega.$
For $m, n \in \{1, \dots, N\}$ and $\x \in \overline \Omega$, define
\begin{align}
    &s_{mn} = \int_0^T \Psi_n'(t)\Psi_m(t) dt, \label{smn}\\ 
    &\mathfrak{f}_m(\x, U, \nabla U) =\int_0^T F\Big(\x, \sum_{n = 1}^N u_n(\x) \Psi_n(t), \sum_{n = 1}^N \nabla u_n(\x) \Psi_n(t)\Big)\Psi_m(t) dt \label{3.7}
\end{align}
where $U = (u_1, u_2, \dots, u_N)^{\rm T}$.
Denote by $S$ and $\mathcal{F}(\x, U, \nabla U)$ the matrix $(s_{mn})_{m, n = 1}^N$ and the vector $(\mathfrak{f}_m(\x, U, \nabla U))_{m = 1}^N$ respectively. 
Since 
\[
    \int_0^T \Psi_n(t)\Psi_m(t) = \left\{
        \begin{array}{ll}
             1 & m = n, \\
             0 & m \not = n, 
        \end{array}
    \right.
\]
it follows from \eqref{3.5} that
\begin{equation}
     \Delta U - SU + \mathcal{F}(\x, U, \nabla U) = 0 
     \quad
     \mbox{for all } \x \in \Omega.
     \label{3.8}
\end{equation}
On the other hand, due to \eqref{data} and \eqref{3.2}, we have for all $m \in \{1, \dots, N\}$ and $\x \in \partial \Omega$,
\begin{align}
    & u_m(\x) = \int_0^T g_0(\x, t) \Psi_m(t)dt, \label{3.9}
    \\
    & \partial u_m(\x) = \int_0^T g_1(\x, t) \Psi_m(t)dt. \label{3.10}
\end{align}
For $\x \in \partial \Omega$, define 
\begin{align}
    &G_0(\x) = \Big(\int_0^T g_0(\x, t) \Psi_m(t)dt\Big)_{m = 1}^N, \label{3.11}
    \\
    &G_1(\x) = \Big(\int_0^T g_1(\x, t) \Psi_m(t)dt\Big)_{m = 1}^N. \label{3.12}
\end{align}
It follows from \eqref{3.8}--\eqref{3.12}, the vector $U$ satisfies the Cauchy like boundary problem
\begin{equation}
    \left\{
    \begin{array}{ll}
         \Delta U - SU + \mathcal{F}(\x, U, \nabla U) = 0 & \x \in \Omega, \\
         U(\x) = G_0(\x)& \x \in \partial \Omega,\\
         \partial_{\nu} U(\x) = G_1(\x) &\x \in \partial \Omega.
    \end{array}
    \right.
    \label{3.13}
\end{equation}

In the next section, we combine a Carleman estimate and the Newton method to compute a solution $U$ to \eqref{3.13}. 
Once this step is done, Problem \ref{isp} is solved.
In fact, having $U$ in hand, we can compute $u(\x, t)$ via \eqref{3.1}. 
The desired function $p(\x)$ is given by $u(\x, 0)$ for all $\x \in \Omega.$

\begin{Remark}
Due to the cut off in \eqref{3.1},  system \eqref{3.13} is not exact. 
We called it an approximation model.
As mentioned in Remark \ref{rem1}, directly solving the inverse problem with the optimization method may be problematic since the stability is just conditionally logarithmic.
In contrast, by a step of approximation, we derive a system of elliptic equations with Cauchy data. It is well-known that solving elliptic equations with full boundary conditions is stable. This promises the success for our numerical study. 
This cut off technique was used to solve several different types of inverse problems; see e.g. \cite{VoKlibanovNguyen:IP2020, Khoaelal:IPSE2021, KhoaKlibanovLoc:SIAMImaging2020, LeNguyen:jiip2022, LeNguyenNguyenPowell:JOSC2021, Nguyen:CAMWA2020, Nguyens:jiip2020, NguyenNguyenTruong:arxiv2022}.
In contrast, proving the convergence of \eqref{3.13} as $N \to \infty$ is extremely challenging. 
Studying the behavior of \eqref{3.13} as $N \to \infty$ is out of the scope of the paper.
\end{Remark}

\section{The Carleman-Newton method}\label{sec_CarNew}

Currently, there are at least three different Carleman based methods to solve \eqref{3.13}:
\begin{enumerate}
    \item {\it The convexification method}.
    The main idea of the convexification method to solve \eqref{3.13} is to minimize the Carleman weighted functional 
    \begin{equation}
        U\mapsto\int_{\Omega} W_\lambda(\x)|\Delta U - SU + \mathcal{F}(\x, U, \nabla U)|^2d\x 
        + \mbox{a regularization term}
        \label{convex}
    \end{equation}
    
    subject to the given boundary conditions for some
    Carleman weight function $W_\lambda(\x)$. 
    With suitable choice of $W_\lambda(\x)$ and the regularization term, one can prove that the functional in \eqref{convex} is strictly convex in any bounded set of $H^s(\Omega)^N$ where $s > \ceil{d/2} + 2.$ The strict convexity implies that the minimizer is unique. Two other important theoretical results for the convexification method are (1) the minimizer can be obtained by using the conventional gradient descent method and the (2) the minimizer is an approximation of the desired solution to \eqref{3.13}. The original idea about the convexification method is introduced in \cite{KlibanovIoussoupova:SMA1995}.
    See \cite{KlibanovNik:ra2017, KhoaKlibanovLoc:SIAMImaging2020, LeNguyen:JSC2022, KlibanovNguyenTran:JCP2022} for follow-up results.
    Although effective in delivering good numerical solutions, the convexification method has a drawback. It is time consuming. 
    We therefore do not employ the convexification method in this paper.

    \item {\it The Carleman contraction method.} 
    The main idea of the contraction method to solve \eqref{3.13} is that we take an arbitrary function $U_0 \in H^2(\Omega)^N$. 
    Assume that $U_n$ is known, we compute $U_{n + 1}$ by solving 
    \begin{equation}
    \left\{
    \begin{array}{ll}
         \Delta U_{n+1} - SU_{n+1} + \mathcal{F}(\x, U_n, \nabla U_{n}) = 0 & \x \in \Omega, \\
         U_{n+1}(\x) = G_0(\x)& \x \in \partial \Omega,\\
         \partial_{\nu} U_{n+1}(\x) = G_1(\x) &\x \in \partial \Omega.
    \end{array}
    \right.
    \label{3.13p}
\end{equation}
    by using the quasi-reversibility method \cite{LattesLions:e1969} combining with a suitable Carleman weight function as in \cite{LeNguyen:jiip2022}.
    One can follow the arguments in \cite{LeNguyen:jiip2022, Nguyen:preprintActa} to prove the convergence of the constructed sequence $\{U_n\}_{n \geq 0}$ to the true solution to \eqref{3.13}. For more details, see the following papers \cite{Le:preprintCONN2022, LeNguyen:jiip2022, Nguyen:preprintActa, NguyenNguyenTruong:arxiv2022}.
    This method was used to solve a similar inverse problem to Problem \ref{isp}. We therefore do not repeat it in this paper.
    \item \label{Newton} The Carleman method combining with linearization \cite{LeNguyenTran:preprint2021}. We name this method Carleman-Newton method. Details of this method will be given in this section.
\end{enumerate}
All three methods above are effective in solving quasi-linear PDEs. In this paper, we use method \ref{Newton}. 
We choose method \ref{Newton} because it is quite powerful in the sense that we have already successfully applied it to compute viscosity solutions to a large class of Hamilton-Jacobi equations in \cite{LeNguyenTran:preprint2021}.

\subsection{A heuristic approach based on the Newton method}\label{sec_heu}

To express the idea behind numerical method \ref{Newton}, we recall here the well-known Newton method to solve nonlinear algebraic equation $f(x) = 0$ for $x \in \R$.
In applying this method, one begins by choosing an initial guess for the solution, say $x_0$. Let $l_0$ be the tangent line to the graph of the function $f$ at $x_0$. One finds the intersection of $l_0$ to the $x-axis$, called $x_1$. 
Let $l_1$ be the tangent line to the graph of the function $f$ at $x_1$. One finds the intersection of $l_1$ to the $x-$axis, called $x_2$. Continuing this procedure, one can obtain a sequence $\{x_n\}_{n \geq 0}$. Under the usual convexity assumptions on $f$, this sequence converges to the solution to the equation $f(x) = 0.$
We combine the idea of the Newton method and a Carleman estimate to solve nonlinear PDEs of the form \eqref{3.13}. 
It is important to mention that this combination is fairly powerful in the sense that the convergence of the constructed sequence to the true solution to \eqref{3.13} is guaranteed regardless of the distance from initial guess to the true solution.

We solve \eqref{3.13} in the strong sense. 
That means we compute a vector valued function $U$ in $H^2(\Omega)^N$ that is the ``best fit" \eqref{3.13}. 
In the analysis, we will use the notation
\begin{equation}
    H = \big\{
        V \in H^2(\Omega)^N: V|_{\partial \Omega} = G_0 \mbox{ and } \partial_{\nu} V|_{\partial \Omega} = G_1
    \big\}.
    \label{setH}
\end{equation}
and
\begin{equation}
    H_0 = \big\{
        V \in H^2(\Omega)^N: V|_{\partial \Omega} = 0 \mbox{ and } \partial_{\nu} V|_{\partial \Omega} = 0
    \big\}.
\end{equation}
The set $H$ is called the set of admissible solutions and the set $H_0$ is clearly a closed subspace of $H^2(\Omega)^N.$
Define the operator
\begin{equation}
    \mathcal{L}(U) = \Delta U - SU + \mathcal{F}(\x, U, \nabla U)
    \label{L}
\end{equation} for all vector valued function $U \in H$.
Heuristically, we want to solve the equation $\mathcal{L} U = 0$ in $H$. 
Let $U_0$ be an arbitrary initial function $U_0 \in H$. 
Based on the Newton method, we set $U_1 = U_0 + h_1 \in H$ where $h_1$ solves
\begin{equation}
\left\{
\begin{array}{ll}	
    \mathcal{L}(U_0) + \Delta h_1 - Sh_1 + D\mathcal{F}(\x, U_0, \nabla U_0)(h_1) = 0 &\x \in \Omega,\\
    h_1 = 0 &\x \in \partial \Omega,\\
    \partial_{\nu} h_1 = 0 &\x \in \partial \Omega.
\end{array}
\right.
    \label{4.3}
\end{equation} 
In \eqref{4.3},
\begin{equation}
	D\mathcal{F}(x, \boldsymbol{b}, A)(h_1) = \Big(\nabla_{\boldsymbol{b}} \mathfrak{f}_m(\x, \boldsymbol{b}, A) \cdot h_1 + \nabla_{A} \mathfrak{f}_m(\x, \boldsymbol{b}, A) : \nabla h_1\Big)_{m = 1}^N
	\label{linearization}
\end{equation}
for all $\boldsymbol{b} \in \R^N$ and $A \in \R^{N \times d}.$
In \eqref{linearization}, the notations $``\cdot"$ and $``:"$ are the usual inner products in $\R^N$ and in $\R^{N \times d}$ respectively.
    Repeating the process replacing $U_0$ by $U_1$, we can find $U_2 \in H$ and then a sequence $\{U_n\}_{n \geq 0} \in H$.
More precisely, for each $n \geq 1$, we define $U_n = U_{n - 1} + h_n$ where $h_n$ solves 
\begin{equation}
\left\{
\begin{array}{ll}	
    \mathcal{L}(U_{n - 1}) + \Delta h_n - Sh_n + D\mathcal{F}(\x, U_{n - 1}, \nabla U_{n - 1})(h_n) = 0 &\x \in \Omega,\\
    h_n = 0 &\x \in \partial \Omega,\\
    \partial_{\nu} h_n = 0 &\x \in \partial \Omega.
\end{array}
\right.
    \label{4.6}
\end{equation}

Due to both Dirichlet and Neumann boundary conditions, both problems \eqref{4.3} and \eqref{4.6} are over-determined. 
They might not have a solution.
However, this difficulty does not lead to any error in our analysis. 
The main reason is that we only need to find their ``best fit" solutions by the Carleman quasi-reversibility method. See Section \ref{sec_Car_quasi} for a brief discussion about the Carleman quasi-reversibility method. 

\begin{Remark}
In general, since we do not impose any special structure for the  nonlinearity $F$, the vector $\mathcal F$ has no special structure, either.
In this general case, there is no guarantee for the convergence of the sequence $\{U_n\}_{n \geq 0}$ to the true solution to \eqref{3.13}, especially when $U_0$ is far away from the true solution to \eqref{3.13}.
However, we have figured out in \cite{LeNguyenTran:preprint2021} that when we employ the Carleman quasi-reversibility method to solve  \eqref{4.6} to construct the sequence $\{U_n\}_{n \geq 0}$, the desired convergence is true.
In Section \ref{sec_convergence}, we  improve convergence result in \cite{LeNguyenTran:preprint2021} in the sense that the new convergence is in the $H^2$ norm while the similar convergence  in \cite{LeNguyenTran:preprint2021} is with respect to the $H^1$ norm.
On the other hand, in this paper we study a noise analysis, which was not done in \cite{LeNguyenTran:preprint2021}.
Since this method is based on a Carleman estimate and is inspired by the Newton method, we name our approach the Carleman-Newton  method.
\end{Remark}

\subsection{A Carleman estimate and the Carleman quasi-reversibility method} \label{sec_Car_quasi}

We recall a Carleman estimate, which serves as an important tool to prove the convergence of the Carleman Newton method to solve problem \ref{3.13}.
We have the following lemma.
\begin{Lemma}[Carleman estimate, see  \cite{LeNguyen:jiip2022}]
	Let $\x_0$ be a point in $\R^d \setminus \overline \Omega$ such that $r(\x) = |\x - \x_0| > 1$ for all $\x \in \Omega$.
	Let $b > \max_{\x \in \overline \Omega} r(\x)$ be a fixed constant.
	There exist positive constants $\beta_0$  depending only on $b$, $\x_0$, $\Omega$ and $d$ such that
	for all function $v \in C^2(\overline \Omega)$ satisfying 
	 \[
	 	v(\x) = \partial_{\nu} v(\x) = 0  \quad \mbox{for all } \x \in \partial \Omega,
	 \]
	the following estimate holds true
\begin{multline}
	\int_{\Omega} e^{2\lambda b^{-\beta} r^{\beta}(\x)}|\Delta v(\x)|^2 d\x
	\geq 
	\frac{C}{\lambda \beta^{7/4} b^{-\beta}} \int_{\Omega}e^{2\lambda b^{-\beta} r^\beta(\x)} r^{2\beta}(\x)|D^2v(\x)|^2 d\x
	\\
	+ 	C \lambda^3 \beta^4 b^{-3 \beta} \int_{\Omega} r^{2\beta}(\x) e^{2\lambda b^{-\beta} r^{\beta}}|v(\x)|^2 d\x
	\\
	+ C \lambda \beta^{1/2} b^{-\beta}\int_{\Omega} e^{2\lambda b^{-\beta} r^{\beta}(\x)} |\nabla v(\x)|^2 d\x
	\label{Car est}
\end{multline}
	for $\beta \geq \beta_0$ and $\lambda \geq \lambda_0$.
		Here, $D^2 v = (v_{x_i x_j} )_{i, j = 1}^d$ is the Hessian matrix of $v$, $\lambda_0 = \lambda_0(b, \Omega, d, \x_0) > 1$ is a positive number with $\lambda_0 b^{-\beta} \gg 1$ and $C = C(b, \Omega, d, \x_0) > 1$ is a constant.  These numbers depend only on listed parameters.
	 \label{carleman estimate 1}
\end{Lemma}	

\begin{Corollary}
	Recall $\beta_0$ and $\lambda_0$ as in Lemma  \ref{carleman estimate 1}.
	Fix $\beta = \beta_0$ and let the constant $C$ depend on $\x_0,$ $\Omega,$ $d$ and $\beta$. 
	There exists a constant $\lambda_0$ depending only on $\x_0,$ $\Omega,$ $d$ and $\beta$ such that
  for all function $v \in H^2(\Omega)$ with 
	\[
		v(\x) = \partial_{\nu} v(\x) = 0 
		\quad \mbox{ on } \partial \Omega,
	\] 
	we have
	\begin{multline}
	\int_{\Omega} e^{2\lambda b^{-\beta} r^{\beta}(\x)}|\Delta v(\x)|^2 d\x
	\geq 
	C\lambda^{-1}\int_{\Omega}e^{2\lambda b^{-\beta} r^\beta(\x)} |D^2 v(\x)|^2 d\x
	\\
	+ 	C \lambda^3 \int_{\Omega}  e^{2\lambda b^{-\beta} r^{\beta}}|v(\x)|^2 d\x
	+ C \lambda \int_{\Omega} e^{2\lambda b^{-\beta} r^{\beta}(\x)} |\nabla v(\x)|^2 d\x	
	\label{33}
\end{multline}
	for all $\lambda \geq \lambda_0$.
	\label{carleman estimate}
\end{Corollary}


We will now explain the Carleman quasi-reversibility method to solve \eqref{4.6}.
Let $\beta_0$ and $\lambda_0$ be as in Corollary \ref{carleman estimate}. Fix $\beta = \beta_0$.
For $\lambda > \lambda_0$,
given a vector valued function $U_{n - 1} \in H$, $n \geq 1$, we say that 
\begin{multline}
    h_n = \underset{\varphi \in H_0}{\rm argmin} \Big[\int_{\Omega}e^{2\lambda b^{-\beta} r^\beta(\x)}\big|\Delta \varphi - S\varphi + D\mathcal{F}(\x, U_{n-1}, \nabla U_{n-1})(\varphi) + \mathcal{L}(U_{n-1}) \big|^2d\x 
    \\
    + \epsilon \|U_{n-1} + \varphi\|_{H^2(\Omega)^N}^2\Big]
    \label{h_define}
\end{multline}
is the solution to \eqref{4.6} due to the Carleman quasi-reversibility method. 
The number $\epsilon \in (0, 1)$ is called the regularization parameter.
The presence of the Carleman weight function $e^{2\lambda b^{-\beta} r^{-\beta}(\x)}$ is the key for us to establish the convergence result in this paper. The important role of this Carleman weight function suggests the name Carleman quasi-reversibility method.
The  quasi-reversibility method without the presence of a Carleman weight function was first introduced in \cite{LattesLions:e1969}.
The convergence of the quasi-reversibility method as $\epsilon \to 0$ was proved in \cite{Nguyen:CAMWA2020}.

\begin{Remark}
The unique minimizer in \eqref{h_define} can be proved by using the same arguments in   \cite[Theorem 4.1]{Le:preprintCONN2022}. For brevity, we do not repeat the proof here.
By using the Carleman quasi-reversibility method, we do not find the exact solution to \eqref{4.6}. 
Rather, we compute the best fit solution $h_n$.
This feature is important because  \eqref{4.6} might not have a solution.
When \eqref{4.6} does have a solution, the reader can find the proof of the convergence of this best fit solution to the true solution as $\epsilon \to 0$ in \cite{Nguyen:CAMWA2020}.
\label{rem4}
\end{Remark}

Inspired by the heuristic arguments in Section \ref{sec_heu}, we propose Algorithm \ref{alg} to numerically solve \eqref{3.13}.
The main aim of this section is to prove the efficiency of this algorithm.
\begin{algorithm}[ht]
\caption{\label{alg}The procedure to compute the numerical solution to \eqref{3.13}}
	\begin{algorithmic}[1]
	\STATE \label{step1_alg} Choose a regularized parameter $0 < \epsilon \ll 1$, a Carleman weight function $e^{2\lambda b^{-\beta }r^{\beta}}$, and a threshold number $0 < \kappa_0 \ll 1.$ 
	\STATE \label{step2_alg}  
	Choose an initial solution $U_0 \in H.$
	\STATE  Set $n = 1$. 
	\STATE \label{h}
	 Let $U_n = U_{n - 1} + h_n$ where $h_n \in H_0$ is the miminizer of $J_{n-1}: H_0 \to \R$ defined as
\[
	J_{n - 1}(\varphi) = \int_{\Omega}e^{2\lambda b^{-\beta} r^\beta(\x)}\big| \Delta \varphi - S\varphi + D\mathcal{F}(\x, U_{n - 1}, \nabla U_{n - 1})(\varphi) + \mathcal{L} (U_{n-1})\big|^2 d\x 
	+ \epsilon \|U_{n - 1} + \varphi\|^2_{H^2(\Omega)}
\]
for all $\varphi \in H_0$.
    \IF{$\|U_{n} - U_{n - 1}\|_{L^{\infty}} \leq \kappa_0$}
		\STATE Reassign $n:=n+1$. 
		\STATE Go back to step \ref{h}.
	\ELSE
	    \STATE \label{s9} Set the computed solution $U^{\rm comp} = U_n.$
	\ENDIF
\end{algorithmic}
\end{algorithm}

\section{The global convergence of the Carleman-Newton method}\label{sec_convergence}

The following result is one of the main theorems of this paper. We first consider the case when $\mathcal F$ has a finite $C^2$ norm.
If $\|\mathcal F\|_{C^2} = \infty$, we can apply the truncating technique in the proof of Theorem \ref{thm_uni} to reduce the problem to the former case, see Remark \ref{rem53} for details.

In this section, we consider the case when the observed data of Problem \ref{isp} are noisy.
As a result, the values of the boundary data of \eqref{3.13} are not exact. 
Denote by $G_{0}^{\delta}$ and $G_1^{\delta}$ the noisy data with the noise level $\delta$. 
It is worth mentioning that equation \eqref{3.13} with $G_0$ and $G_1$ replaced by $G_{0}^{\delta}$ and $G_1^{\delta}$, respectively, might not have a solution.
Let $G_{0}^{*}$ and $G_1^{*}$ be the noiseless versions of $G_0$ and $G_1$ respectively. 
By noise level $\delta$, we mean that
\begin{equation}
\inf \big\{\mathcal{E} \in H^2(\Omega)^N: \mathcal{E}|_{\partial \Omega} = G_0^{\delta} - G_0^* 
\mbox{ and } 
\partial_{\nu}\mathcal{E}|_{\partial \Omega} = G_1^{\delta} - G_1^*
\big\}
    \leq \delta.
    \label{noise}
\end{equation}
Due to \eqref{noise}, there is an error vector value function $\mathcal{E}$ satisfying
    \begin{equation}
    \left\{
        \begin{array}{l}
             \|\mathcal{E}\|_{H^2(\Omega)} < 2\delta,  \\
             \mathcal{E}|_{\partial \Omega} = G_0^\delta - G_0^*,\\ 
             \partial_{\nu}\mathcal{E}|_{\partial \Omega} = G_1^\delta - G_1^*.
        \end{array}
    \right.
    \label{5,2}
    \end{equation}
\begin{Remark}
  The existence of the error function $\mathcal E$ in \eqref{5,2} implies that the noise must be the restriction of a vector valued function in $H^2(\Omega)$ onto $\partial \Omega$. Hence, the noise must be in $H^{3/2}(\partial \Omega)$.
  This assumption might not be realistic since the noise in measurement might be not smooth.
  There are several techniques to smooth  out the noise; for e.g., the Tikhonov method and the b-spline method.
  However, in this paper, we do not have to apply one of these techniques to obtain numerical results. The proposed method works with nonsmooth data. 
  That means, the Carleman-Newton method is stronger than what we can rigorously prove. 
\end{Remark}    
    
Given the noisy data, the set of admissible solutions $H$, defined in \eqref{setH}, becomes
\begin{equation}
     H^{\delta} = \big\{
        V \in H^2(\Omega)^N: V|_{\partial \Omega} = G_0^\delta \mbox{ and } \partial_{\nu} V|_{\partial \Omega} = G_1^\delta
    \big\}.
\end{equation}
Since $G_{0}^{*}$ and $G_1^{*}$ contain no noise, we can assume that
\begin{equation}
    \left\{
    \begin{array}{ll}
         \Delta U^* - SU^* + \mathcal{F}(\x, U^*, \nabla U^*) = 0 & \x \in \Omega, \\
         U^*(\x) = G_0^*(\x)& \x \in \partial \Omega,\\
         \partial_{\nu} U^*(\x) = G_1^*(\x) &\x \in \partial \Omega
    \end{array}
    \right.
    \label{quasi-system}
\end{equation}
has a unique solution $U^*$.

\begin{Theorem}
    Assume that the source function $p$ is in the class $\mathcal{P}$ and assume that $\|\mathcal F\|_{C^2} < \infty$.
     Let $U_0$ be a vector valued function in $H^\delta$.
     Let $\{U_n\}_{n \geq 0}$ be the sequence constructed in Algorithm \ref{alg}.
     Assume that \eqref{quasi-system} has a unique solution $U^*$. 
     For $\lambda > \lambda_0$ and $\beta = \beta_0$ as in Corollary \ref{carleman estimate}, we have
     \begin{multline}
     \int_{\Omega} e^{2\lambda b^{-\beta} r^\beta(\x)} \big(\lambda^{-2} |D^2 (U_n - U^*)|^2 
	+ 	 |U_n - U^*|^2 
	+  |\nabla (U_n - U^*)|^2\big) d\x
	\\
	\leq
	\Big(\frac{C}{\lambda}\Big)^{n+1}\int_{\Omega} e^{\lambda b^{-\beta} r^{\beta}(\x)}  \big(  \lambda^{-2}|D^2 (U_0 - U^*)|^2 + |U_0 - U^*|^2 + |\nabla (U_0 - U^*)|^2\big) d\x 
	\\
	+ \Big(\frac{C}{\lambda}\Big)^{n+1}\int_{\Omega} e^{\lambda b^{-\beta} r^{\beta}(\x)}  \big(  \lambda^{-2}|D^2 \mathcal E|^2 + |\mathcal E|^2 + |\nabla \mathcal E|^2\big) d\x 
	\\
		+ \frac{C/\lambda}{1 - C/\lambda}
		\Big[
		\int_{\Omega} e^{2\lambda b^{-\beta} r^{\beta}(\x)} \big|
	    \Delta \mathcal{E} - S\mathcal{E}
	\big|^2d\x +
		\epsilon\|\mathcal E\|_{H^2(\Omega)^2}^2
		+\epsilon\|U^*\|_{H^2(\Omega)^2}^2
		\Big]
		\label{est_conv}
\end{multline}
     where $C$ is a constant depending only on $b,$ $\Omega$, $d$ and $\x_0$.
     \label{thm2}
\end{Theorem}

\begin{Corollary}
    Fix $\lambda > \lambda_0$ such that $\theta = C/\lambda \in (0, 1)$. 
    It follows from \eqref{est_conv} that
    \begin{equation}
        \|U_n - U^*\|_{H^2(\Omega)^N}^2 \leq \theta^{n+1} \|U_0 - U^*\|_{H^2(\Omega)^N}^2 + \frac{C\theta}{1 - \theta} (\delta^2 + \epsilon \|U^*\|_{H^2(\Omega)^N}^2). 
    \label{5,6}
    \end{equation}
    Hence, the sequence $\{U_n\}_{n \geq 0}$ strongly converges to $U^*$ in the $H^2$ norm regardless of the initial distance from $U_0$ to $U^*$. The error caused by the noise and the regularization technique is  $O(\delta + \sqrt{\epsilon}).$
    \label{col5.1}
\end{Corollary}

\begin{Corollary}
    Let $U^{\rm comp} = U_n$ where $n$ is as in Step \ref{s9} of Algorithm \ref{alg}.
    Write $U^{\rm comp} = (u_1^{\rm comp}, \dots, u_N^{\rm comp})^{\rm T}$ and $U^* = (u^*_1, \dots, u^*_N)^{\rm T}.$
    Define 
    \begin{equation}
        p^{\rm comp}(\x) =  \sum_{m = 1}^N u^{\rm comp}_m(\x)\Psi_m(0)
        \mbox{ and }
        p^*(\x ) = \sum_{m = 1}^N u^*_m(\x)\Psi_m(0)
        \label{pcomp}
    \end{equation} for all $\x \in \Omega$.
    Due to \eqref{5,6}, we have
    \begin{equation}
        \|p^{\rm comp} - p^*\|_{H^2(\Omega)^N}^2 \leq C\theta^{n+1} \|U_0 - U^*\|_{H^2(\Omega)^N}^2 + \frac{C\theta}{1 - \theta} (\delta^2 + \epsilon \|U^*\|_{H^2(\Omega)^N}^2). 
    \label{5,7}
    \end{equation}
    Estimate \eqref{5,7} guarantee the convergence of Algorithm \ref{alg}.
    \label{col52}
\end{Corollary}
  
\begin{Remark}
Due to \eqref{5,7}, the convergence of the method is $O(\theta^n)$ as $n \to \infty,$ which is exponentially fast.
Hence, the computational cost is not expensive.
On the other hand, the error in computation is $O(\delta + \sqrt{\epsilon})$ as $(\delta, \epsilon) \to (0, 0).$
\end{Remark}

\begin{Remark}
  The proof of Theorem \ref{thm2} is similar to that of \cite[Theorem 4.1]{LeNguyenTran:preprint2021} with suitable modifications. 
    These modifications are to improve the quality of the convergence. 
    In fact, the new convergence is with respect to the $H^2$ norm while the convergence in \cite[Theorem 4.1]{LeNguyenTran:preprint2021} is with respect to the $H^1$ norm. 
    Moreover, no noise analysis was done in \cite[Theorem 4.1]{LeNguyenTran:preprint2021} while the Lipschitz stability with respect to the noise is shown in \eqref{5,6}.
\end{Remark}
\begin{proof}[Proof of Theorem \ref{thm2}]
  
    Fix $n \geq 1$. 
    Since $h_n$ defined in Step \ref{h} of Algorithm \ref{alg} is the minimizer of $J_{n - 1}$ in $H_0$, by the variational principle, the Fr\'echet derivative  $DJ_{n - 1}(h_{n})(\varphi) $ vanishes for all $\varphi \in H_0$.
    We have
    \begin{multline}
        \int_{\Omega} e^{2\lambda b^{-\beta} r^{\beta}(\x)}\big[
            \Delta h_n - Sh_n + D\mathcal{F}(\x, U_{n-1}, \nabla U_{n-1})(h_n) + \mathcal{L}(U_{n-1})
        \big]
        \cdot
        \big[
            \Delta \varphi - S\varphi 
            \\
            + D\mathcal{F}(\x, U_{n-1}, \nabla U_{n-1})(\varphi)
        \big] d\x
        + \epsilon \langle U_{n - 1} + h_n, \varphi\rangle_{H^2(\Omega)^N} = 0
        \label{4,10}
    \end{multline}
    for all $\varphi \in H_0.$
    Due to the definition of $U_n$ in Step \ref{h} of Algorithm \ref{alg},  we have $h_n = U_n - U_{n-1}$. This, together with the definition of the operator $\mathcal{L}$ in \eqref{L} and the identity \eqref{4,10}, gives
    \begin{multline}
        \int_{\Omega} e^{2\lambda b^{-\beta} r^{\beta}(\x)}\big[
            \Delta U_n - SU_n + \mathcal F(\x, U_{n-1}, \nabla U_{n-1}) + D\mathcal{F}(\x, U_{n-1}, \nabla U_{n-1})(U_{n} - U_{n-1}) 
        \big]
        \\
        \cdot
        \big[
            \Delta \varphi - S\varphi + D\mathcal{F}(\x, U_{n-1}, \nabla U_{n-1})(\varphi)
        \big] d\x
        + \epsilon \langle U_{n}, \varphi\rangle_{H^2(\Omega)^N} = 0
        \label{4.11}
    \end{multline}
    for all $\varphi \in H_0.$
    On the other hand, since $U^*$ is a solution to \eqref{quasi-system}, we have
    \begin{equation}
        \int_{\Omega} e^{2\lambda b^{-\beta} r^{\beta}(\x)}\big[
            \Delta U^* - SU^* + \mathcal F(\x, U^*, \nabla U^*) 
        \big]
        \cdot
        \big[
            \Delta \varphi - S\varphi + D\mathcal{F}(\x, U_{n-1}, \nabla U_{n-1})(\varphi)
        \big] d\x = 0
        \label{4.12}
    \end{equation}
    for all $\varphi \in H_0.$
    Combining \eqref{4.11} and \eqref{4.12}, we obtain
    \begin{multline}
      \int_{\Omega} e^{2\lambda b^{-\beta} r^{\beta}(\x)}\Big[
            \Delta (U_n - U^*)
            - S(U_n - U^*) 
            \\
            + \mathcal{F}(\x, U_{n - 1}, \nabla U_{n - 1}) + D\mathcal{F}(\x, U_{n - 1}, \nabla U_{n - 1})(U_n  - U_{n - 1}) - \mathcal{F}(\x, U^*, \nabla U^*)
        \Big]
        \\
        \cdot
        \Big[
            \Delta \varphi - S\varphi 
            + D\mathcal{F}(\x, U_{n-1}, \nabla U_{n-1})(\varphi)
        \Big] d\x
        + \epsilon \langle U_{n} , \varphi\rangle_{H^2(\Omega)^N} = 0.
        \label{4.13}
    \end{multline}
    
    Let $\varphi_i = U_i - U^* - \mathcal{E} \in H_0$ for all $i \geq 1$ where $\mathcal E$ is the vector in \eqref{5,2}. 
    Using the test function $\varphi = \varphi_n$ in \eqref{4.13} gives
    \begin{multline}
    	 \int_{\Omega} e^{2\lambda b^{-\beta} r^{\beta}(\x)}\Big[
            \Delta (\varphi_n + \mathcal{E}) - S(\varphi_n + \mathcal{E}) + D\mathcal{F}(\x, U_{n - 1}, \nabla U_{n - 1})(\varphi_n)
            + \mathcal{F}(\x, U_{n - 1}, \nabla U_{n - 1}) 
            \\- \mathcal F(\x, U^*, \nabla U^*) 
            - D\mathcal F(\x, U_{n - 1}, \nabla U_{n - 1})(\varphi_{n - 1})
            \Big] 
            \cdot
            \Big[
            \Delta \varphi_n - S\varphi_n + D\mathcal{F}(\x, U_{n - 1}, \nabla U_{n - 1})(\varphi_n)
            \Big] d\x
            \\
        + \epsilon \langle \varphi_{n} + \mathcal{E} + U^*, \varphi_n\rangle_{H^2(\Omega)^N} = 0.
        \label{4.14}
    \end{multline}

    It follows from \eqref{4.14} that
    \begin{multline}
    	\int_{\Omega} e^{2\lambda b^{-\beta} r^{\beta}(\x)} \Big| 
		\Delta \varphi_n - S\varphi_n + D\mathcal{F}(\x, U_{n-1}, \nabla U_{n-1})(\varphi_n)
	\Big|^2d\x 
	+ \epsilon\|\varphi_n\|_{H^2(\Omega)^N}^2
	\\
	+ \epsilon\langle \mathcal{E}, \varphi_n\rangle_{H^2(\Omega)^N}
	+ \epsilon\langle U^*, \varphi_n\rangle_{H^2(\Omega)^N}
	= \int_{\Omega} e^{2\lambda b^{-\beta} r^{\beta}(\x)} \Big[
		\mathcal F(\x, U^*, \nabla U^*) - \mathcal{F}(\x, U_{n - 1}, \nabla U_{n - 1})  
		\\
		+ D\mathcal F(\x, U_{n - 1}, \nabla U_{n - 1})(\varphi_{n - 1})
	\Big] 
	\cdot \Big[ 
		\Delta \varphi_n - S\varphi_n + D\mathcal{F}(\x, U_{n-1}, \nabla U_{n-1})(\varphi_n)
	\Big]d\x 
	\\
	+
	\int_{\Omega} e^{2\lambda b^{-\beta} r^{\beta}(\x)} \Big[
		\Delta \mathcal{E} - S\mathcal{E}
	\Big] 
	\cdot \Big[ 
		\Delta \varphi_n - S\varphi_n + D\mathcal{F}(\x, U_{n-1}, \nabla U_{n-1})(\varphi_n)
	\Big]d\x. 
	\label{5.6}
	   \end{multline}
	   Using the inequality $|ab| \leq  a^2 + \frac{1}{4}b^2$ and \eqref{5.6}, we have
	   \begin{multline}
    	\int_{\Omega} e^{2\lambda b^{-\beta} r^{\beta}(\x)} \Big| 
		\Delta \varphi_n - S\varphi_n + D\mathcal{F}(\x, U_{n-1}, \nabla U_{n-1})(\varphi_n)
	\Big|^2d\x
	+ \epsilon\|\varphi_n\|_{H^2(\Omega)^N}^2
	+ \epsilon\langle \mathcal{E}, \varphi_n\rangle_{H^2(\Omega)^N}
	\\
	+ \epsilon\langle U^*, \varphi_n\rangle_{H^2(\Omega)^N}
	\leq  
	C\int_{\Omega} e^{2\lambda b^{-\beta} r^{\beta}(\x)} \Big|
		\mathcal F(\x, U^*, \nabla U^*) - \mathcal{F}(\x, U_{n - 1}, \nabla U_{n - 1}) 
		\\
		+ D\mathcal F(\x, U_{n - 1}, \nabla U_{n - 1})(\varphi_{n - 1})
	\Big|^2d\x
	+ C \int_{\Omega} e^{2\lambda b^{-\beta} r^{\beta}(\x)} \big|
	    \Delta \mathcal{E} - S\mathcal{E}
	\big|^2d\x.
	\label{5.7}
	 \end{multline}
	 Recall the inequality $(a - b)^2 \geq \frac{1}{2} a^2 - b^2$. We have
 \begin{multline}
    	\int_{\Omega} e^{\lambda b^{-\beta} r^{\beta}(\x)} \big| 
		\Delta \varphi_n - S\varphi_n + D\mathcal{F}(\x, U_{n-1}, \nabla U_{n-1})(\varphi_n)
	\big|^2d\x 
	\\
	\geq
	\frac{1}{2}\int_{\Omega} e^{\lambda b^{-\beta} r^{\beta}(\x)} \Big| 
		\Delta \varphi_n\Big|^2d\x
		 -
		\int_{\Omega} e^{\lambda b^{-\beta} r^{\beta}(\x)}\Big| - S\varphi_n + D\mathcal{F}(\x, U_{n-1}, \nabla U_{n-1})(\varphi_n)
	\Big|^2d\x. 
	\label{5.8}
\end{multline}
On the other hand, since $\|\mathcal F\|_{C^2} < \infty$, we can find a constant $C$ such that
\begin{multline}
    \Big|
		\mathcal F(\x, U^*, \nabla U^*) - \mathcal{F}(\x, U_{n - 1}, \nabla U_{n - 1})  + D\mathcal F(\x, U_{n - 1}, \nabla U_{n - 1})(\varphi_{n - 1}) 
	\Big|
	\\
	\leq C\big(|\varphi_{n - 1}| + |\nabla \varphi_{n - 1}|\big),
	\label{5.9}
\end{multline}
and 
\begin{equation}
    \Big| - S\varphi_n + D\mathcal{F}(\x, U_{n-1}, \nabla U_{n-1})(\varphi_n)
	\Big|
	\leq C\big(|\varphi_{n}| + |\nabla \varphi_{n}|\big)
	\label{5.10}
\end{equation}
for all $\x \in \overline \Omega.$
Combining \eqref{5.7}, \eqref{5.8}, \eqref{5.9}, \eqref{5.10} and the inequality $|ab| \leq \frac{1}{2}a^2 + \frac{1}{2}b^2$ gives
 \begin{multline*}
\int_{\Omega} e^{2\lambda b^{-\beta} r^{\beta}(\x)} \Big| 
		\Delta \varphi_n\Big|^2d\x + \epsilon \|\varphi_n\|_{H^2(\Omega)^N}^2
		\leq C\int_{\Omega} e^{\lambda b^{-\beta} r^{\beta}(\x)}  (|\varphi_{n }|^2 + |\nabla \varphi_{n }|^2 + |\varphi_{n - 1}|^2 
		\\+ |\nabla \varphi_{n - 1}|^2) d\x 
		+C \int_{\Omega} e^{2\lambda b^{-\beta} r^{\beta}(\x)} \big|
	    \Delta \mathcal{E} - S\mathcal{E}
	\big|^2d\x
		+ C\epsilon\|U\|_{H^2(\Omega)^2}^2
		+ C\epsilon\|\mathcal{E}\|_{H^2(\Omega)^2}^2
		+C\epsilon \|\varphi_n\|_{H^2(\Omega)^N}^2,
\end{multline*}
which implies
 \begin{multline}
\int_{\Omega} e^{\lambda b^{-\beta} r^{\beta}(\x)} \Big| 
		\Delta \varphi_n\Big|^2d\x 
		\leq C\int_{\Omega} e^{\lambda b^{-\beta} r^{\beta}(\x)}  (|\varphi_{n }|^2 + |\nabla \varphi_{n }|^2 + |\varphi_{n - 1}|^2 + |\nabla \varphi_{n - 1}|^2) d\x 
		\\
		+C \int_{\Omega} e^{2\lambda b^{-\beta} r^{\beta}(\x)} \big|
	    \Delta \mathcal{E} - S\mathcal{E}
	\big|^2d\x
		+ C\epsilon\|\mathcal E\|_{H^2(\Omega)^2}^2
		+ C\epsilon\|U^*\|_{H^2(\Omega)^2}^2.
		\label{5.11}
\end{multline}

We now apply the Carleman estimate in Corollary \ref{carleman estimate}. 
Using \eqref{33} for each component of the vector $\varphi_n$, we have
\begin{multline}
	\int_{\Omega} e^{2\lambda b^{-\beta} r^{\beta}(\x)}|\Delta \varphi_n|^2 d\x
	\geq 
	C\lambda^{-1}\int_{\Omega}e^{2\lambda b^{-\beta} r^\beta(\x)} |D^2 \varphi_n|^2 d\x
	\\
	+ 	C \lambda^3 \int_{\Omega}  e^{2\lambda b^{-\beta} r^{\beta}}|\varphi_n|^2 d\x
	+ C \lambda \int_{\Omega} e^{2\lambda b^{-\beta} r^{\beta}(\x)} |\nabla \varphi_n|^2 d\x	
	\label{5.12}
\end{multline}
	Using \eqref{5.11} and \eqref{5.12}, we have
\begin{multline}
	\lambda^{-1}\int_{\Omega}e^{2\lambda b^{-\beta} r^\beta(\x)} |D^2 \varphi_n|^2 d\x
	+ 	\lambda^3 \int_{\Omega}  e^{2\lambda b^{-\beta} r^{\beta}}|\varphi_n|^2 d\x
	+ \lambda \int_{\Omega} e^{2\lambda b^{-\beta} r^{\beta}(\x)} |\nabla \varphi_n|^2 d\x
	\\
	\leq 
	C\int_{\Omega} e^{\lambda b^{-\beta} r^{\beta}(\x)}  (|\varphi_{n }|^2 + |\nabla \varphi_{n }|^2 + |\varphi_{n - 1}|^2 + |\nabla \varphi_{n - 1}|^2) d\x 
		\\
		 +C \int_{\Omega} e^{2\lambda b^{-\beta} r^{\beta}(\x)} \big|
	    \Delta \mathcal{E} - S\mathcal{E}
	\big|^2d\x
		+ C\epsilon\|\mathcal E\|_{H^2(\Omega)^2}^2
		+ C\epsilon\|U^*\|_{H^2(\Omega)^2}^2.
	\label{5.13}
\end{multline}
Since $\lambda$ is large, we can write \eqref{5.13} as
\begin{multline}
	\lambda^{-1}\int_{\Omega}e^{2\lambda b^{-\beta} r^\beta(\x)} |D^2 \varphi_n|^2 d\x
	+ 	\lambda^3 \int_{\Omega}  e^{2\lambda b^{-\beta} r^{\beta}}|\varphi_n|^2 d\x
	+ \lambda \int_{\Omega} e^{2\lambda b^{-\beta} r^{\beta}(\x)} |\nabla \varphi_n|^2 d\x
	\\
	\leq 
	C\int_{\Omega} e^{\lambda b^{-\beta} r^{\beta}(\x)}  (|\varphi_{n - 1}|^2 + |\nabla \varphi_{n - 1}|^2) d\x 
		+C \int_{\Omega} e^{2\lambda b^{-\beta} r^{\beta}(\x)} \big|
	    \Delta \mathcal{E} - S\mathcal{E}
	\big|^2d\x
	\\
		+ C\epsilon\|\mathcal E\|_{H^2(\Omega)^2}^2
		+ C\epsilon\|U^*\|_{H^2(\Omega)^2}^2.
	\label{5.14}
\end{multline}
It follows from \eqref{5.14} that
\begin{align}
    \int_{\Omega} e^{2\lambda b^{-\beta} r^\beta(\x)} &\big(\lambda^{-2} |D^2 \varphi_n|^2 
	+ 	 |\varphi_n|^2 
	+  |\nabla \varphi_n|^2\big) d\x
	\nonumber
	\\
	&\leq 
	\frac{C}{\lambda}\int_{\Omega} e^{\lambda b^{-\beta} r^{\beta}(\x)}  \big(  |\varphi_{n - 1}|^2 + |\nabla \varphi_{n - 1}|^2\big) d\x 
		+ \frac{1}{2\lambda}\epsilon\|U\|_{H^2(\Omega)^2}^2
		\nonumber
		\\
	&\leq \frac{C}{\lambda}\int_{\Omega} e^{\lambda b^{-\beta} r^{\beta}(\x)}  \big(  \lambda^{-2}|D^2 \varphi_{n-1}|^2 + |\varphi_{n - 1}|^2 + |\nabla \varphi_{n - 1}|^2\big) d\x \nonumber
	\\
	&\hspace{2cm}
		+ \frac{C}{\lambda}\Big[
		\int_{\Omega} e^{2\lambda b^{-\beta} r^{\beta}(\x)} \big|
	    \Delta \mathcal{E} - S\mathcal{E}
	\big|^2d\x +
		\epsilon\|\mathcal E\|_{H^2(\Omega)^2}^2
		+\epsilon\|U^*\|_{H^2(\Omega)^2}^2
		\Big].\label{5.15}
\end{align}
It follows from \eqref{5.15} and by induction, we have
\begin{multline}
     \int_{\Omega} e^{2\lambda b^{-\beta} r^\beta(\x)} \big(\lambda^{-2} |D^2 \varphi_n|^2 
	+ 	 |\varphi_n|^2 
	+  |\nabla \varphi_n|^2\big) d\x
	\\
	\leq
	\Big(\frac{C}{\lambda}\Big)^{n+1}\int_{\Omega} e^{\lambda b^{-\beta} r^{\beta}(\x)}  \big(  \lambda^{-2}|D^2 \varphi_{0}|^2 + |\varphi_{0}|^2 + |\nabla \varphi_{0}|^2\big) d\x 
	\\
		+ \frac{C/\lambda}{1 - C/\lambda}
		\Big[
		\int_{\Omega} e^{2\lambda b^{-\beta} r^{\beta}(\x)} \big|
	    \Delta \mathcal{E} - S\mathcal{E}
	\big|^2d\x +
		\epsilon\|\mathcal E\|_{H^2(\Omega)^2}^2
		+\epsilon\|U^*\|_{H^2(\Omega)^2}^2
		\Big]
		\label{5.16}
\end{multline}
Recall $\varphi_i = U_i - U^* - \mathcal E$, $i \geq 0$. Using \eqref{5.16} and inequalities $(a - b)^2 \geq \frac{1}{2}a^2 - b^2$ and $(a + b)^2 \leq 2a^2 + 2b^2$, we obtain \eqref{est_conv}.
\end{proof}

Theorem \ref{thm2} and Corollary \ref{col52} suggest Algorithm \ref{alg2} to solve the inverse source problem under consideration.

\begin{algorithm}[ht]
\caption{\label{alg2}The procedure to compute the numerical solution to Problem \ref{isp}}
	\begin{algorithmic}[1]
	\STATE\label{step1_alg2} Choose a basis $\{\Psi_n\}_{n \geq 1}$ of $L^2(0, T)$. Choose a cut-off number $N > 0$. 
	See Section \ref{sec6.1} and Figure \ref{fig_chooseN} for a reasonable choice of $N$.
	\STATE Compute the vector valued functions $G_0$ and $G_1$ as in \eqref{3.11} and \eqref{3.12}.
	\STATE \label{step3_alg2} Apply Algorithm \ref{alg} to compute a numerical solution $U^{\rm comp}$ to \eqref{3.13}.
	\STATE  \label{step4_alg2}
	Reconstruct the source function by the first formula in \eqref{pcomp}.
\end{algorithmic}
\end{algorithm}


\begin{Remark}[The case when $\|\mathcal F\|_{C^2} = \infty$]
	The condition $\|\mathcal F\|_{C^2} < \infty$ in Theorem \ref{thm2} is too strong. It can be relaxed in the context when the forward problem has a unique and bounded solution as in Assumption \ref{Assump}.
	Assume that $p^* \in \mathcal P$, see Definition \ref{defP} for the definition of $\mathcal P$.
	Then, we can apply the truncation technique as in the proof of Theorem \ref{thm_uni} to compute $U^*.$
	Since solution of the forward problem is bounded as in Assumption  \ref{Assump}, it follows from \eqref{3.2} that the true solution $U^*$ to \eqref{quasi-system} is bounded, namely,
	\[
		|U^*| + |\nabla U^*| < B
	\] for some number $B$. 
	Define $\chi_B$ as in \eqref{chi_B} and $\mathcal{F}_B = \chi_B \mathcal F$. Clearly, the vector valued function $U^*$ satisfies
	\begin{equation}
    \left\{
    \begin{array}{ll}
         \Delta U^* - SU^* + \mathcal{F}_B(\x, U^*, \nabla U^*) = 0 & \x \in \Omega, \\
         U^*(\x) = G_0^*(\x)& \x \in \partial \Omega,\\
         \partial_{\nu} U^*(\x) = G_1^*(\x) &\x \in \partial \Omega
    \end{array}
    \right.
    \label{quasi-system_truncate}
\end{equation}
Since $\mathcal {F}_B$ has a bounded $C^2$ norm, we can apply Algorithm \ref{alg} for \eqref{quasi-system_truncate} to compute $U^*.$
\label{rem53}
\end{Remark}
\section{Numerical Simulations}\label{NS}
In this section, we will illustrate the theoretical results by some numerical examples.
For simplicity, we implement Algorithms \ref{alg} and \ref{alg2} in 2D and in the finite difference scheme. 
The set $\Omega$ is the square $(-R,R)^2$ where $R=1$. We solve the forward problem on a larger domain $\Omega_1:=(-R_1,R_1)^2$ where $R_1=6$ 
We will use the notation $\textbf{x}=(x,y)$ to represent points on $\Omega$. 
We arrange a uniform $N_\x^1 \times N_\x^1$ grid, named as $\mathcal{G}$, on $\Omega_1$ as follows
\begin{align*}
\mathcal{G}_1=\big\{ (x_i,y_j): x_i=-R_1+(i-1)h_{\textbf{x}}, y_j=-R_1+(j-1)h_{\textbf{x}}, i, j = 1, \dots, N_\x\big\}
\end{align*} where $N_\x^1 = 240$ and
$h_{\textbf{x}}=2R_1/(N_\x^1-1)$. 
We also choose $T=1.5$ and divide the interval $[0,T]$ uniformly into $N_t = 3000$ uniform subintervals. 
For the forward data generation, we solve \eqref{main_eqn} by the explicit method. 
We then collect the (noiseless) data $g_0(\textbf{x},t)=u(\textbf{x},t)$ and $g_1(\textbf{x},t)=\partial_{\nu}u(\textbf{x},t)$ on the lateral boundary $\partial \Omega_T$. Noise is then added to the data using the following expression:
{\begin{equation}
    g_i^{\delta}(x,t)=g_i(x,t)[1+\delta (-1+2\eta(x,t))], i=0,1
    \label{add_noise}
\end{equation}} where $\delta$ denotes the noise level and $\eta(x,t)$ generates uniformly distributed numbers in the interval $[-1,1]$. 
The implementation of Algorithm \ref{alg2}  to reconstruct the source function from the noisy data is given in the next subsection. We would also like to remark here that the inversion algorithm has been implemented on the uniform grid $\mathcal{G} = \mathcal{G}_1 \cap \overline{\Omega}.$


\subsection{Implementation}\label{sec6.1}
We  now describe the steps to implement Algorithm \ref{alg2} which enable us to reconstruct the source function in this inverse source problem. 
\par \textit{Step 1:} 
Following the arguments laid out in Section \ref{sec_method}, (also see \cite[Section 5.2]{LeNguyen:jiip2022}), in Step \ref{step1_alg2} of Algorithm \ref{alg2}, we choose the special basis functions $\{\Psi_n\}_{n \geq 1}$ of the space $L^2(0,T)$, which was originally introduced by Klibanov in \cite{Klibanov:jiip2017}.
We recall that Klibanov's basis is obtained by applying the Gram-Schmidt orthogonalization process for  the sequence of functions $\{\Phi_n(t)\}_{n\geq 1}$ where $\Phi_n(t)=t^{n-1} e^{t-T/2}$. 
A comparative advantage of this special basis over traditional trigonometric  Fourier basis is that the first term of the trigonometric basis being a constant, vanishes on taking the derivatives. 
As such, there will be no contribution from the coefficient of the first basis function on the left hand side of equation \eqref{3.4}, i.e. the information from $u_1(x)\Psi_1^{\prime}(t)$ is lost, while this information is retained if we use Klibanov's basis. 
To choose the cut off number $N$,
we take the data $u(\x, t)$ for $\x \in \Gamma$ and $t \in (0, T)$ where
\[
    \Gamma = \{\x = (x, y = R): |x| \leq R\} \subset \partial \Omega.
\]
We then compute the function
\[
    e_N(\x, t) = |u(\x, t) - \sum_{n = 1}^N  u_n(\x)\Psi_n(t)|
\]
where $u_n$ is as in \eqref{3.2}. This step of choosing $N$ by evaluating $e_N$ for different values of $N$ is  illustrated in Figure \ref{fig_chooseN}. Here, $u$ solves equation \eqref{main_eqn} with $p$ given as in test 2 below.
\begin{figure}[h]
    \centering
    \subfloat[$N = 10$]{\includegraphics[width=0.3\textwidth]{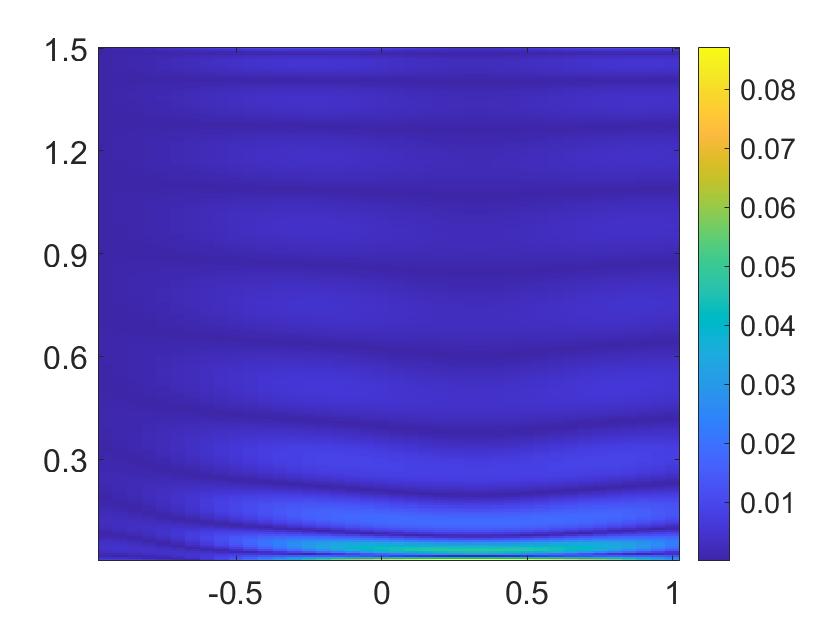}}
    \quad
    \subfloat[$N = 20$]{\includegraphics[width=0.3\textwidth]{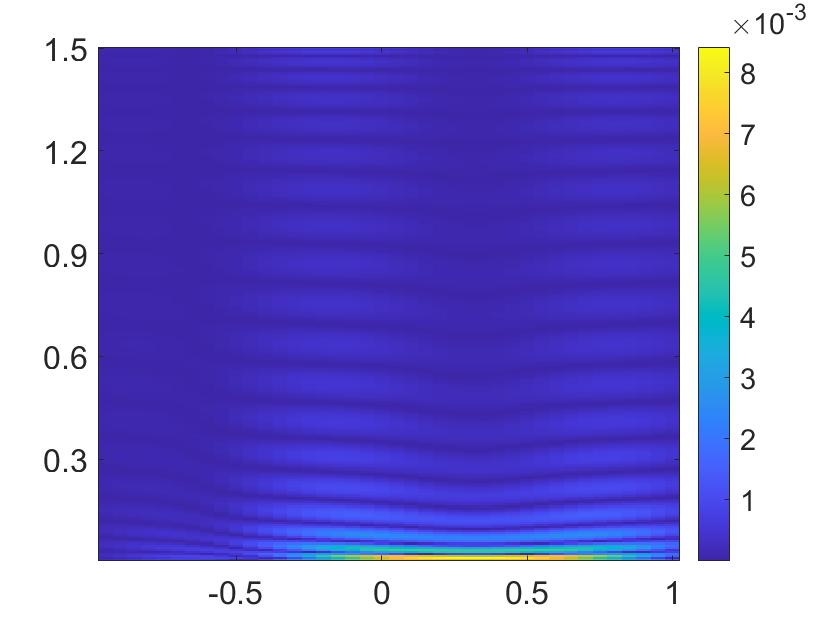}}
    \quad
    \subfloat[$N = 35$]{\includegraphics[width=0.3\textwidth]{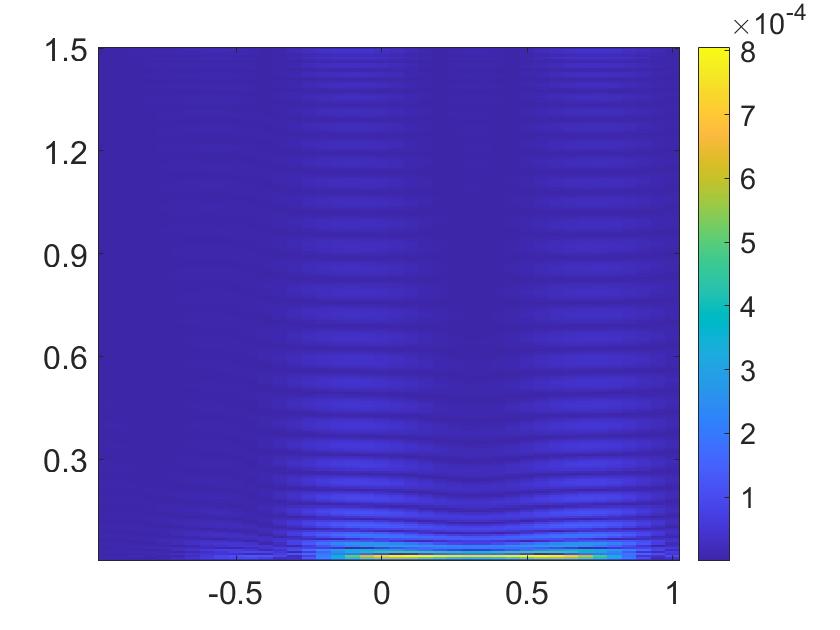}}
    \caption{\label{fig: Choose N}The graphs of the functions $e_N(\x, t)$, $(\x, t) \in \Gamma \times (0, T)$. In the figure, the x-axis runs from $-1$ to $1$ and the y-axis is from $0$ to $1.5$. It is evident that when $N = 35$, the error function $e_N$ is sufficiently small, say $\|e_N\|_{L^{\infty}} < 5\times 10^{-2}$.} 
    \label{fig_chooseN}
\end{figure}
\par \textit{Step 2:}{ Once the basis terms $\{\Psi_m\}_{m=1}^N$ are found out, one can easily evaluate the integrals given by equations \ref{3.9} and \ref{3.10}. The computation of the (vector) data term as per equations \eqref{3.11} and \eqref{3.12} is straight forward.}
\par \textit{Step 3:} The implementation of Step \ref{step3_alg2} of Algorithm \ref{alg2} or equivalently Algorithm \ref{alg} follows closely from that in \cite{LeNguyenTran:preprint2021}.
In Step \ref{step1_alg} of Algorithm \ref{alg}, we need to choose the artificial parameters.
The parameters for the Carleman weight function used in step 4 of Algorithm \ref{alg} are given by $r=\lvert \x-\x_0\rvert$ where $\x_0=(0,1.5)$, $b=5, \lambda=40$ and $\beta=10$. The regularization parameter used was $\epsilon=10^{-7}$. 
These values are chosen by a manual trial and error process. We use test 1 below as the reference test. We modify $\epsilon,$  $\lambda$, $b$, $\beta,$  $\x_0$ until we obtain a satisfactory solution with noiseless data. 
These parameters are used for all other tests with all noise levels.
In Step \ref{step2_alg} of Algorithm \ref{alg}, we compute
 the initial term of the sequence $\{U_n\}_{n \geq 0},$ i.e., $U_0$. 
 This function is found by solving the problem  \eqref{3.13} wherein we use only the linear {part of the term $\mathcal{F}(x,U,\nabla U)$ denoted by $\mathcal{F}_{lin}(U)$. The (vector) term $\mathcal{F}_{lin}(U)$ is computed from \eqref{3.7} by evaluating for each $m$ in the following way: Suppose the term $F(x,u,\nabla u)$ appearing in the integral in \eqref{3.7} is given by $F(x,u,\nabla u)= F_{lin}(u)+\text{non-linear terms}$. For example, for the first test given in the next subsection, $F(x,u,\nabla u)=u(1-u)=\underbrace{u}_{\text{linear part}}-\underbrace{u^2}_{\text{non-linear part}}$. Similarly, for the second test, $F(x,u,\nabla u)=\underbrace{u}_{linear part}+\underbrace{\sqrt{(\lvert\nabla u\rvert^2+1)}}_{non-linear part}$. Thus for both the tests, we take $F_{lin}(u)=u$ which is then used in \eqref{3.7} in place of $F(x,u,\nabla u) $ to get $\mathcal{F}_{lin}(U)$. Now to generate the initial guess we solve the following system:
 \begin{equation}
    \left\{
    \begin{array}{ll}
         \Delta U - SU + \mathcal{F}_{lin}(U) = 0 & \x \in \Omega, \\
         U(\x) = G_0(\x)& \x \in \partial \Omega,\\
         \partial_{\nu} U(\x) = G_1(\x) &\x \in \partial \Omega.
    \end{array}
    \right.
    \label{mod_eq}
\end{equation}}
{Next, the minimizer of the corresponding (strictly) convex linear functional $J_n(\phi)$ appearing in Step 4 of algorithm 1 can be found by solving the corresponding Normal equation. For brevity, we do not describe the finite difference implementation of various operators appearing in the expression for the Normal equation corresponding to  $J_n(\phi)$ as these are similar to those discussed elsewhere, see e.g. \cite [section 5.3]{LeNguyen:jiip2022}. 
We briefly mention that the matrix forms of all such operators is necessarily stored as sparse matrices. Solution to this Normal equation is then found by using the MATLAB function `lsqlin' which in turn is the
 minimizer of the functional $J_n(\phi)$ on $H_0$. }
We let the iterative Algorithm 1 run for six iterations, i.e. $n=6$ as it was observed that $\|U_{6} - U_{5}\|_{L^{\infty}}$ was small enough in all numerical experiments. 
\par \textit{Step 4: }From the computed solution $U^{comp}=U_6$, one can reconstruct the source function as described in Step \ref{step4_alg2} of Algorithm \ref{alg2}.

\subsection{Numerical examples}


\textbf{Test 1.} The true source function is given by:
\[
    p_{true}= 
\begin{cases}
    8,& \text{if } x^2+(y-0.3^2)\leq 0.45^2\\
    0,              & \text{otherwise}
\end{cases}
\]
The nonlinearity considered in this case is given by: 
\begin{align*}
    F(x,u,\nabla u)=u(1-u).
\end{align*}
\begin{figure}[ht]
	\centering
\subfloat[The true source function]{\includegraphics[width=0.3\textwidth]{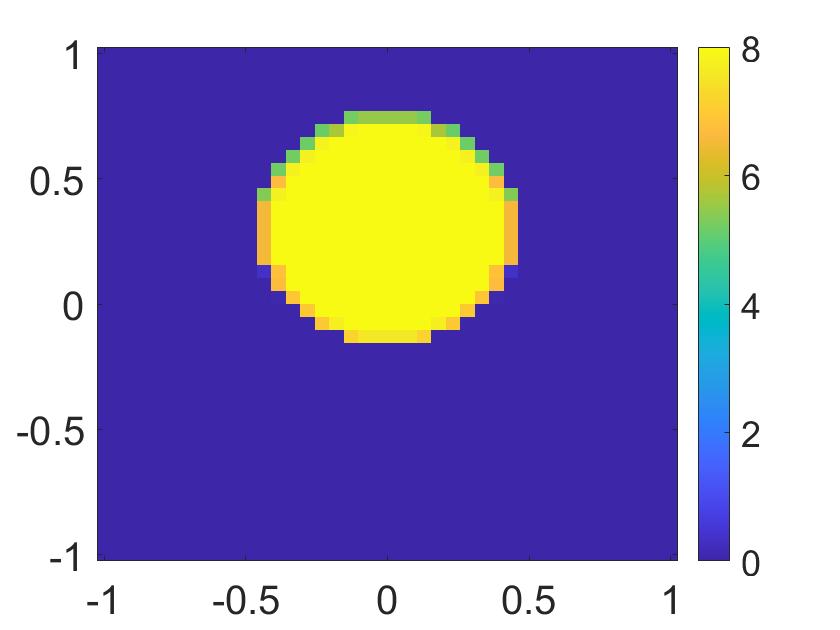}}
    \quad
    \subfloat[The computed source function]{\includegraphics[width=0.3\textwidth]{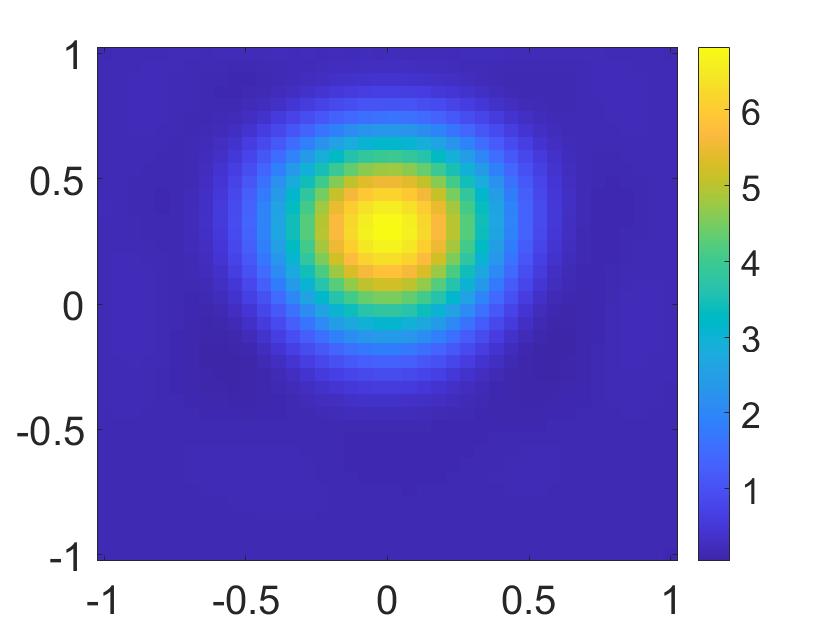}}
    \quad
    \subfloat[The difference of $U_n$ and $U_{n - 1}$, $n = \overline{1, 6}.$ ]{\includegraphics[width=0.3\textwidth]{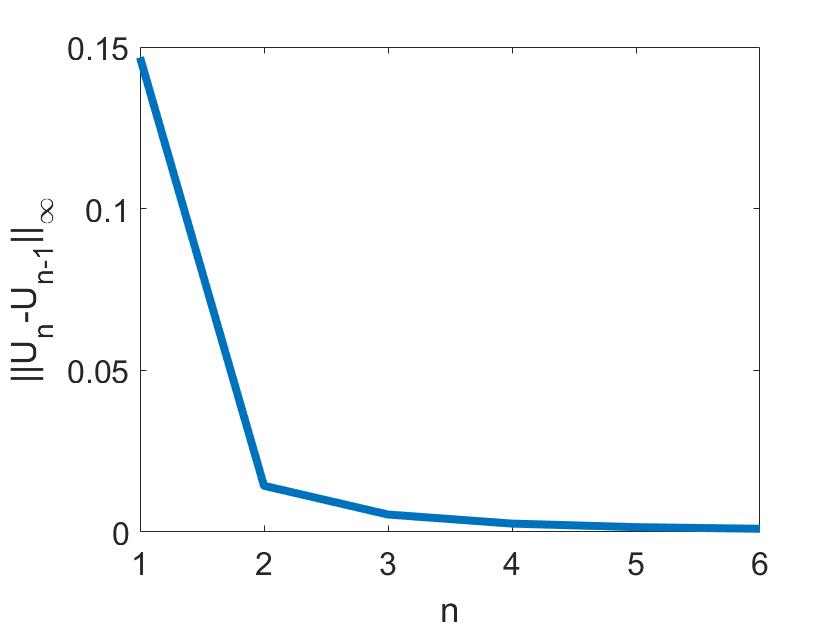}}
 \caption{Test 1. (a) contains the true phantom $p_{true}$. (b) is the reconstruction of the phantom from noisy data after 6 iterations. (c) shows the convergence of the method by plotting $||U_n-U_{n-1}||_{\infty}$ against the number of iterations $n$. Noise level used in this experiment was 20\%.}
 \label{mod1}
 \end{figure}
 
 \noindent We notice that the position of the inclusion is correctly identified. Furthermore, the maximal computed value in the reconstruction at $20\%$ noise level is $6.89$ for a relative error of $13.87\%$ in the reconstruction.\\
 
 
 \textbf{Test 2.} The true source function is given by:
\[
    p_{true}= 
\begin{cases}
    12,& \text{if } (x-0.5)^2+(y-0.5)^2\leq 0.35^2\\
    10,& \text{if } (x+0.5)^2+(y+0.5)^2\leq 0.35^2\\
    14,& \text{if } (x-0.5)^2+(y+0.5)^2\leq 0.35^2\\
    9,& \text{if } (x-0.5)^2+(y-0.5)^2\leq 0.35^2\\
    0,              & \text{otherwise}
\end{cases}
\]
The nonlinearity considered in this case is given by: 
\begin{align*}
    F(x,u,\nabla u)=u+\sqrt{(\lvert\nabla u\rvert^2+1)}.
\end{align*}
\begin{figure}[ht]
	\centering
	\subfloat[The true source function]{\includegraphics[width=0.3\textwidth]{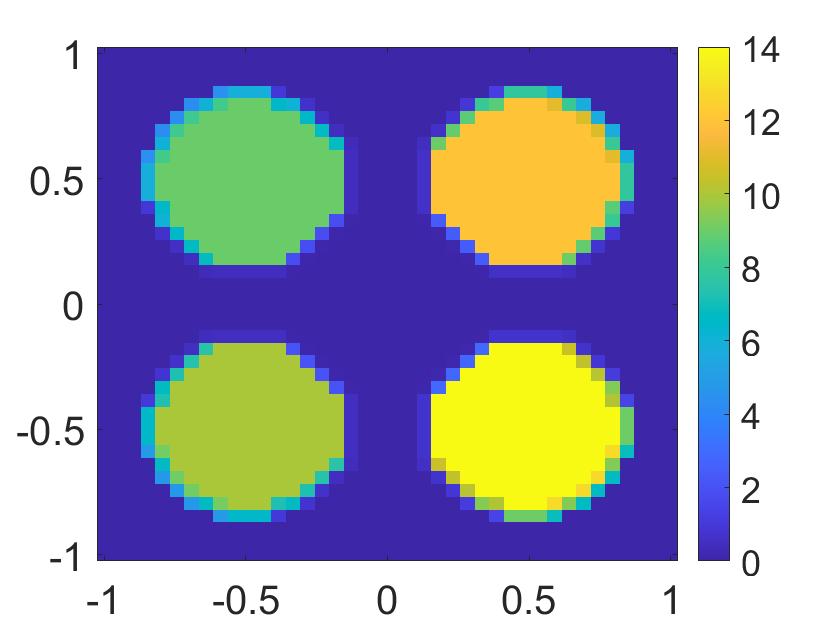}}
    \quad
    \subfloat[The computed source function]{\includegraphics[width=0.3\textwidth]{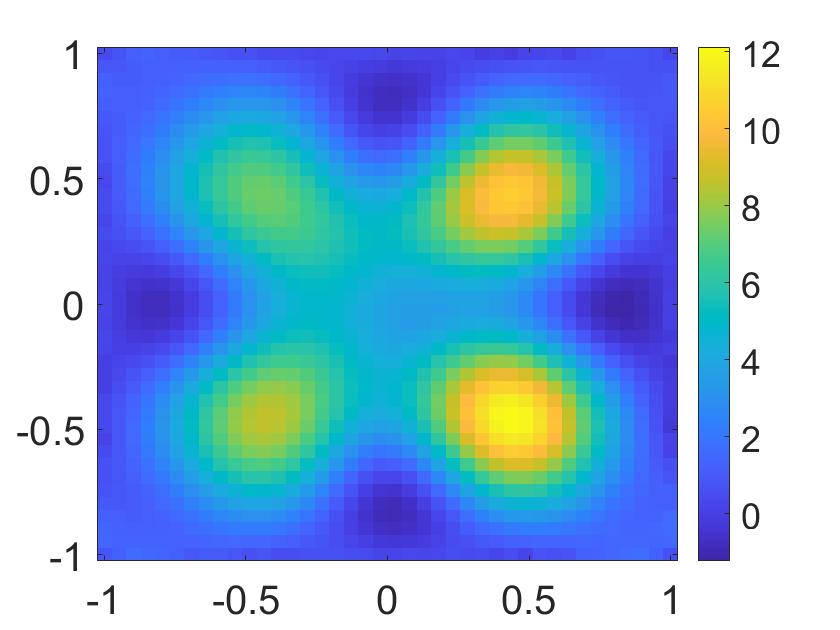}}
    \quad
    \subfloat[The difference of $U_n$ and $U_{n - 1}$, $n = \overline{1, 6}.$]{\includegraphics[width=0.3\textwidth]{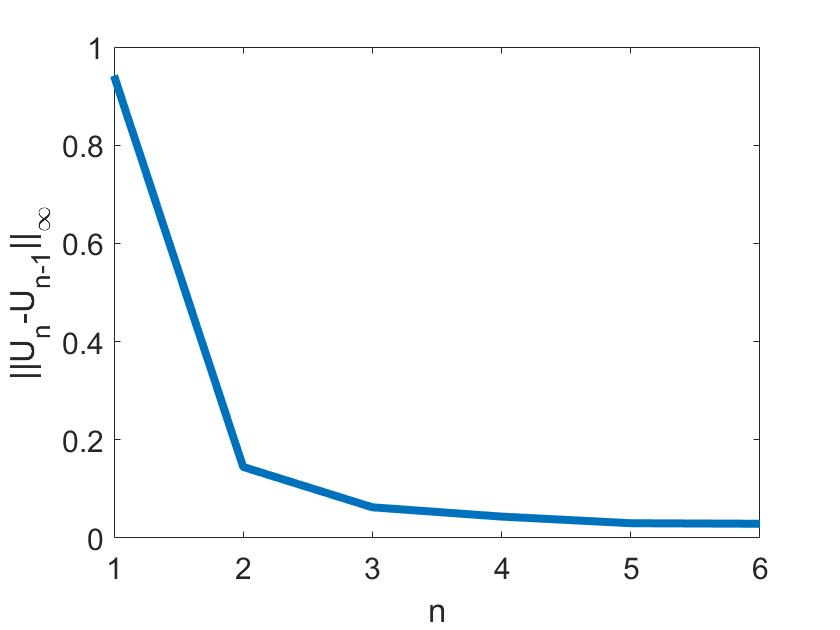}}
 \caption{Test 2. (a) contains the true phantom $p_{true}$. (b) is the reconstruction of the phantom from noisy data after 6 iterations. (c) shows the convergence of the method by plotting $||U_n-U_{n-1}||_{\infty}$ against the number of iterations $n$. Noise level used in this experiment was 20\%.}
 \label{mod2}
 \end{figure}
 Again, we notice, that the positions of all four inclusions are correctly identified. The maximal computed values in the reconstructions of the four inclusions at $20\%$ noise level are $10.66$ (up, right), $12.01$ (down, right), $7.33$ (up, left) and $8.60$ (down, left), for relative errors of $11.16\%$, $13.57\%$, $8.37\%$ and $14\%$ respectively.
\begin{Remark}
In our numerical experiments, we tried several noise levels from $1\%$ to $20\%$ with the quality of reconstruction not deteriorating by any appreciable amount at higher levels of noise. We have provided the reconstructions at $20\%$ noise level here. This shows a very high degree of the robustness of our reconstructions compared to the noise-level. This fact is also borne out by figure \ref{noise_1} in which we have compared the quality of reconstruction across the pixels lying on the vertical line passing through $x=0.5$ for the phantom in Test 2.

We recall two main facts to interpret the strong stability of our method with respect to noise, see \eqref{add_noise}. 
The first one is that we have truncated all high frequency components of the data while  the lower frequency components are not sensitive with the noise.
The second fact is that we have reduced the inverse problem to the problem of solving elliptic equations given boundary data, which is known to be stable.  The stability is guaranteed by Theorem \ref{thm2}, Corollary \ref{col5.1} and the estimate \eqref{5,6}.
\end{Remark}

\begin{figure}[ht]
	\centering
\includegraphics[width=0.5\textwidth]{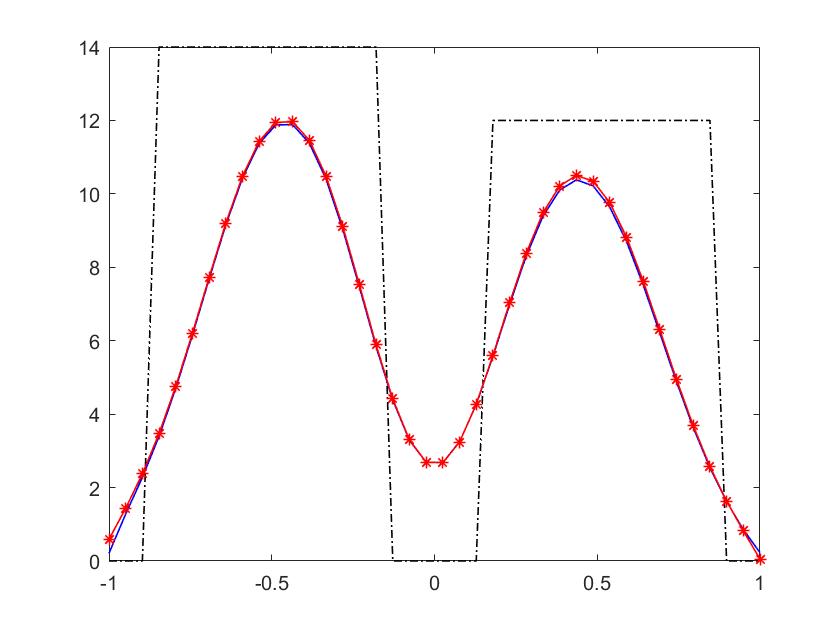}\hspace{0.01cm}

 \caption{ The true phantom (dashed, black), reconstruction at 5\% noise level (red, -*) and reconstruction at 20\% noise level (solid, blue) on the vertical line passing through $x=0.5$ in phantom for test 2. The horizontal axis in the plot refers to the y-axis of the phantom and the vertical axis gives the pixel values. Reconstruction at 20\% noise level is only very slightly worse than that at 5\% noise level.}
 \label{noise_1}
 \end{figure}

 It is remarkable that in the tests above, the value of the source inside the inclusions are high. Hence, the locally convergent approaches based on optimization, which require a good initial guess of the true solution, might not be applicable. 
 Unlike this, our methods provide reliable solutions.
 
 \section{Concluding remarks}  \label{sec7}
 
 In this paper, we introduce an iterative method to solve an inverse source problem for a nonlinear parabolic equations.
 This method can be considered as the combination of Carleman estimates and the Newton method in solving nonlinear equations.
In the first step, we truncate the Fourier series of the solution to the governing equation to
derive a system of nonlinear elliptic equations. The computed solution to this system yields directly the solution for the inverse source problem. 
In order the compute such a solution, we repeatedly solve the linearization of this system.
By using Carleman estimate, we proved that the sequence of obtained solutions to the desired solution.
 The strength of our numerical method is that it quickly provides a good approximation to the true source function. Furthermore, it does not require any knowledge of the true source function. This means a good initial guess is not necessary.

The method was implemented in finite difference. Numerical results were shown.

 \section*{Acknowledgement}
The works of TTL and LHN were partially supported  by National Science Foundation grant DMS-2208159, and  by funds provided by the Faculty Research Grant program at UNC Charlotte Fund No. 111272.

\end{document}